\documentclass[leqno,preprint,1p,times]{elsarticle}

\usepackage{amsmath,amsthm,amssymb,amsfonts,ascmac,bm,empheq,mathtools,latexsym,color}
\newtheorem{theorem}{Theorem}[section]
\newtheorem{lemma}[theorem]{Lemma}
\newtheorem{prop}[theorem]{Proposition}
\newtheorem{corollary}[theorem]{Corollary}
\theoremstyle{definition}

\newtheorem{example}[theorem]{Example}

\newtheorem{rmk}[theorem]{Remark}
\newtheorem{remark}[theorem]{Remark}

\numberwithin{equation}{section}

\journal{Journal of Approximation Theory}

\begin{document}

\begin{frontmatter}


\title{Exceptional Bannai-Ito polynomials\tnoteref{label1}}
\tnotetext[label1]{Partially supported by JSPS KAKENHI (Grant Number 16K13761).}

\author[LY]{Yu Luo\corref{cor1}}
\cortext[cor1]{Corresponding author.}
\ead{luo.yu.68e@st.kyoto-u.ac.jp}

\author[TS]{Satoshi Tsujimoto}
\ead{tsujimoto.satoshi.5s@kyoto-u.jp}

\address[LY,TS]{Department of Applied Mathematics and Physics, Graduate School of Informatics, Kyoto University, Yoshida-Honmachi, Sakyo-ku, Kyoto 606-8501, Japan}

\begin{abstract}
\par
We construct a nontrivial type of 1-step exceptional Bannai-Ito polynomials which satisfy a discrete orthogonality 
by using a generalized Darboux transformation. 
In this generalization, the Darboux transformed Bannai-Ito operator is directly obtained through an intertwining relation.
Moreover, the seed solution, which consists of a gauge factor and a polynomial part, 
plays an important role in the construction of these 1-step exceptional Bannai-Ito polynomials. 
And we show that there are 8 classes of gauge factors. 
We also provide the eigenfunctions of the corresponding multiple-step exceptional Bannai-Ito operator 
which can be expressed as a $3\times 3$ determinant. 
\end{abstract}

\begin{keyword}
Exceptional orthogonal polynomials \sep Dunkl shift operator \sep
Bannai-Ito polynomials \sep a generalized Darboux transformation

\MSC[2010] primary 33C45 \sep 33C47 \sep secondary 42C05
\end{keyword} 

\end{frontmatter}

\section{Introduction}
\subsection{Bannai-Ito polynomials}
\par
The Bannai-Ito polynomials, originally introduced in \cite{BI}, 
are recently classified as a new kind of ``classical" orthogonal polynomials \cite{DopBI} 
since they are identified to be the eigenfunctions of the difference operator
\begin{equation}
\label{eq:opBI}
H=\alpha(x)(R-I)+\beta(x)(TR-I).
\end{equation}
$H$ is a Dunkl shift operator, where $R$ is the reflection operator, $T$ is the forward shift operator, 
and $I$ is the identity operator acting as
$$
R[f(x)]=f(-x), \quad T[f(x)]=f(x+1), \quad I[f(x)]=f(x).
$$
Throughout this paper 
the operators $R$, $T$ and $I$ only influence $x$, for example, 
$R[f(x+1)]=f(-x+1)$ not $f(-x-1)$, and $T[f(-x)]=f(-x-1)$ not $f(-x+1)$.
It is known that the Dunkl shift operator has polynomial eigenfunctions of all degrees if and only if 
its coefficients are given by
\begin{equation}
\label{eq:BIcoes}
\alpha(x)=\frac{(x-\rho_1)(x-\rho_2)}{-2x}, \quad \beta(x)=\frac{(x-r_1+1/2)(x-r_2+1/2)}{2x+1},
\end{equation}
where $r_1$, $r_2$, $\rho_1$, $\rho_2$ are real numbers.
In this setting the Dunkl shift operator is called the Bannai-Ito operator. 
The eigenfunctions of the Bannai-Ito operator are the Bannai-Ito polynomials $B_n(x)$,  
$$
H[B_n(x)]
=\alpha(x)(B_n(-x)-B_n(x))+\beta(x)(B_n(-x-1)-B_n(x))
=\lambda_nB_n(x), 
$$
where $B_n(x)$ is of degree $n$, and the eigenvalues are given by 
\vspace{-1mm}
\begin{empheq}[left={\lambda_n=\empheqlbrace}]{align}
\frac{n}{2},\hspace{3.3cm} & \quad \text{ if } n \text{ is even,}  \\
r_1+r_2-\rho_1-\rho_2-\frac{n+1}{2}, & \quad \text{ if } n \text{ is odd.} 
\end{empheq}
One should notice that the Bannai-Ito polynomials are defined upon four parameters, 
i.e. $B_n(x):=B_n(x;\rho_1,\rho_2,r_1,r_2)$, we use the former for simplicity. 

\par
The Bannai-Ito polynomials are discrete orthogonal polynomials. 
They are orthogonal with respect to a discrete, positive measure of weight function $\omega(x)$ 
on the Bannai-Ito grid $\{x_s\}^{N-1}_{s=0}$:
\vspace{-1mm}$$
\sum^{N-1}_{s=0}\omega(x_s)B_n(x_s)B_m(x_s)=h_n\delta_{nm} \hspace{2mm} 
(h_n>0; \hspace{2mm}  0\leq n,m<N),
$$
where $x_s$ ($s=0,1,\ldots,N-1$) are the simple roots of $B_{N}(x)$ \cite{DopBI}.

\subsection{Exceptional orthogonal polynomials}
\par
The last decade has witnessed exciting developments in a new class of generalization of the 
classical orthogonal polynomials (COPs),
which were given a kindly confusing name, 
the ``exceptional" orthogonal polynomials (XOPs). 
It is necessary to remark that the XOPs are a generalization of the COPs, 
since they have the latter as a special case. 
Here we say a polynomial $P_n(x)$ is of degree $n$ if the highest degree of its monomial is $n$. 
In most cases the XOPs have ``gaps" in their degree sequences, 
i.e. there are a finite number of missing degrees in their polynomial sequences 
(while the degree sequences of COPs are $\{0,1,2,\ldots\}$).
Extensive interest have been devoted to many aspects of the theory of XOPs.
See, for instance, \cite{GKM09} for the introduction of the exceptional flag which gives birth to XOPs,  
and \cite{GKM10,EH,XLJ,XOP} for a systematic way of constructing XOPs satisfying 2nd-order differential equations 
(or 2nd-order difference equations in \cite{ECH,EML,EHJ}),
as well as \cite{BEOP} for a complete classification of the (continuous) XOPs.

\par
In the construction of XOPs, the Darboux transformation (DT) plays an important role. 
It was further clarified that multiple-step or higher order DTs lead to XOPs 
labelled by multi-indices \cite{2DT,MDBT,MIOP}.
The 1-step (rational) DT was conducted on a 2nd-order differential operator 
$T=p(x)D_{xx}+q(x)D_{x}+r(x)$ and acts as $T\mapsto T^{(1)}$:
\begin{equation}
\label{eq:DT}
T=\mathcal{B}\circ\mathcal{A}+\lambda, \quad
T^{(1)}=\mathcal{A}\circ\mathcal{B}+\lambda,
\end{equation}
where $\mathcal{A}$, $\mathcal{B}$ are 1st-order differential operators \cite{GKM10}.
An immediate consequence of (\ref{eq:DT}) can be derived as the following intertwining relations
$$
\mathcal{A}\circ T=T^{(1)}\circ\mathcal{A}, \quad
T\circ\mathcal{B}=\mathcal{B}\circ T^{(1)},
$$
which imply that the eigenvalue problem $T[y]=\lambda y$ is equivalent to 
the eigenvalue problem $T^{(1)}[y^{(1)}]=\lambda y^{(1)}$ where $y^{(1)}=\mathcal{A}[y]$.
Thus with a well selected ``seed solution" $\phi$ such that $T[\phi]=\mu\phi$ and $\mathcal{A}[\phi]=0$, 
it simply deletes degrees in the eigenfunction sequence $\{y^{(1)}\}$ of the 
Darboux transformed operator $T^{(1)}$.

\par
The purpose of this paper is to establish an XOPs extension of the Bannai-Ito polynomials 
by constructing the exceptional Bannai-Ito operators which have polynomial eigenfunctions 
of all but a finite number of degrees.
In Section 2, we apply a generalized DT on the Dunkl shift operator. 
As we have just addressed, the traditional DTs factorize a 2nd-order differential/difference 
operator into two 1st-order differential/difference operators, then a new 2nd-order operator can be obtained 
by exchanging these two 1st-order operators. 
However, the Dunkl shift operator $H$ is a 1st-order difference operator, 
which means that the traditional DTs work no more on this operator. 
Our idea is to start from the intertwining relations
$$
\mathcal{F}_{\phi}\circ H=H^{(1)}\circ\mathcal{F}_{\phi}, \hspace{2mm}
H\circ\mathcal{B}_{\phi}=\mathcal{B}_{\phi}\circ H^{(1)},
$$
where the operator $\mathcal{F}_{\phi}$ is a Dunkl shift operator satisfying $\mathcal{F}_{\phi}[\phi(x)]=0$, 
and $\phi(x)$ is an eigenfunction of $H$. 
Using this result, we are able to provide the 1-step exceptional Bannai-Ito operator explicitly. 
We also give more general results on the intertwining relations 
regarding the multiple-step exceptional Bannai-Ito operator, 
and derive the eigenfunctions of the latter thereafter. 

\par
In Section 3, we derive a special class of eigenfunctions of the Bannai-Ito operator, 
which are called the quasi-polynomial eigenfunctions. 
A quasi-polynomial eigenfunction consists of a gauge factor and a polynomial part. 
We show that there are 8 classes of gauge factors by comparing the coefficients of the 
conjugated Bannai-Ito operator and those of the Bannai-Ito operator.
From these quasi-polynomial eigenfunctions, in the next section, 
we will select seed solutions for the DT of the Bannai-Ito operator.

\par
In Section 4, we construct the 1-step exceptional Bannai-Ito polynomials and show that they are orthogonal 
with respect to a discrete measure on the exceptional Bannai-Ito grid. 
Interestingly enough, the degree sequences of the exceptional Bannai-Ito polynomials demonstrate different 
rules compared with the known 1-step XOPs. 
The positivity of the weight functions related to these 1-step exceptional Bannai-Ito polynomials is also considered, 
and we provide some sufficient conditions with respect to certain parameters.

\par
Finally, in Section 5, we give a concluding remark concerning the fact that the exceptional Bannai-Ito polynomials 
constructed in this paper are not the unique ones. 
The reason is that the auxiliary operator $\mathcal{F}_{\phi}$ which plays a crucial role in the construction of the 
exceptional Bannai-Ito polynomials is not unique. 
Some other candidates of this auxiliary operator are also provided, 
while the general cases shall be leaving as an open problem.

\section{Generalized Darboux transformation of Dunkl shift operator}
\par
It has been addressed before that the DT is one of the most powerful tools for constructing XOPs. 
The rational DTs (or the Darboux-Crum transformations) factorize a 2nd-order 
differential/difference operator into two 1st-order differential/difference operators, 
hence the exceptional operator can be obtained by exchanging the two 1st-order operators.
Unfortunately, in our case the Dunkl shift operator is a 1st-order difference operator, 
which makes it difficult to apply the rational DTs directly. 
In this section we present a generalized DT of the Dunkl shift operator, 
and then give exceptional Bannai-Ito operators explicitly.

\par
Consider the eigenvalue problem of the Dunkl shift operator
\vspace{-0.5mm}$$
H[\phi(x)]=\alpha(x)(R-I)[\phi(x)]+\beta(x)(TR-I)[\phi(x)]=\mu\phi(x),
$$
where
$\alpha(x)$, $\beta(x)$ are functions in $x$. 
Then we choose an eigenfunction $\phi(x)$ (not necessarily to be polynomial) of $H$ as a seed solution and let
\vspace{-0.2mm}\begin{eqnarray}
\chi(x) \hspace{-2.2mm}&=&\hspace{-2.2mm}(I-R)[\phi(x)]=\phi(x)-\phi(-x), \label{def: X} \\
\tilde{\chi}(x) \hspace{-2.2mm}&=&\hspace{-2.2mm} (I+TR)[\beta(-x-1)\phi(x)]=\beta(-x-1)\phi(x)+\beta(x)\phi(-x-1). \label{def: XX}
\end{eqnarray}
The functions $\chi(x)$ and $\tilde{\chi}(x)$ satisfy the following properties
$$
\chi(-x)=-\chi(x), \quad \tilde{\chi}(-x-1)=\tilde{\chi}(x),
$$
and the eigenvalue problem $H[\phi(x)]=\mu\phi(x)$ can be written as
\begin{equation}
\label{eq:2.3}
-\alpha(x)\chi(x)+\tilde{\chi}(x)=(\mu+\beta(x)+\beta(-x-1))\phi(x).
\end{equation}
Substituting $x$ by $-x-1$, the above equation becomes
\begin{equation}
\label{eq:2.4}
-\alpha(-x-1)\chi(-x-1)+\tilde{\chi}(x)=(\mu+\beta(-x-1)+\beta(x))\phi(-x-1),
\end{equation}
the operation $(\ref{eq:2.3})\cdot\beta(-x-1)+(\ref{eq:2.4})\cdot\beta(x)$ then implies
\begin{equation}
\label{eq:2.5}
\alpha(x)\beta(-x-1)\chi(x)+\alpha(-x-1)\beta(x)\chi(-x-1)+\mu\tilde{\chi}(x)=0.
\end{equation}
After substituting $x$ by $-x$ in $(\ref{eq:2.3})$ and then subtracting the result by $(\ref{eq:2.3})$ we have
\begin{equation}
\label{eq:2.6}
\tilde{\chi}(-x)-\tilde{\chi}(x)=
-\chi(x)(\mu+\alpha(x)+\alpha(-x)+\beta(x)+\beta(-x-1)),
\end{equation}
which holds under the condition
\begin{equation}
\label{eq:condition b}
\beta(-x)+\beta(x-1)=\beta(x)+\beta(-x-1).
\end{equation}
The properties of $\chi(x)$ and $\tilde{\chi}(x)$ are very important, 
as well as the equations (\ref{eq:2.3}), (\ref{eq:2.4}), (\ref{eq:2.5}), (\ref{eq:2.6}) and the condition (\ref{eq:condition b}),
which will be used frequently in the DT with respect to the Dunkl shift operator $H$. 

\par
Now we define the operator 
\begin{equation}
\label{eq:opF}
\mathcal{F}_{\phi}=\frac{1}{\chi(x)}(R-I)+\frac{1}{\tilde{\chi}(x)}(TR+I)\beta(-x-1)
\end{equation}
such that $\mathcal{F}_{\phi}[\phi(x)]=0$. 
It is easily seen that if the intertwining relation 
\begin{equation}
\label{eq:single Inter}
\mathcal{F}_{\phi}\circ H=H^{(1)}\circ \mathcal{F}_{\phi}
\end{equation}
holds for a new Dunkl shift operator $H^{(1)}$, 
then for any eigenfunction $\psi(x)$ ($\neq\phi(x)$) of $H$ with eigenvalue $\nu$ ($\neq\mu$)), one has
$$
H^{(1)}\big[\mathcal{F}_{\phi}[\psi(x)]\big]=
\mathcal{F}_{\phi}\circ H[\psi(x)]=\nu\cdot\mathcal{F}_{\phi}[\psi(x)].
$$
The equation with respect to the most left and the most right terms 
means that $\mathcal{F}_{\phi}[\psi(x)]$ is the eigenfunction of $H^{(1)}$ with eigenvalue $\nu$.
Before deriving the explicit expression of the operator $H^{(1)}$ 
we first give the following lemma for the convenience of calculation.
\begin{lemma} The operators $R$, $T$ and $I$ satisfy 
\vspace{-2mm}$$RR=I, \quad TRTR=I, \quad RTRT=I.$$
Moreover, for any function $f(x)$, it holds that 
\vspace{-1mm}
\begin{eqnarray*}
& &(R-I)f(x)(R-I) = (f(x)+f(-x))(I-R), \\
& &(TR+I)f(x)(TR+I) = (f(x)+f(-x-1))(TR+I), \\
& &(TR+I)f(x)(TR-I) = (f(x)-f(-x-1))(TR-I), \\
& &(TR-I)f(x)(TR+I) = (f(-x-1)-f(x))(TR+I), \\
& &(R-I)f(x)(TR-I) = f(-x)RTR-f(x)TR-f(-x)R+f(x)I, \\
& &(R-I)f(x)(TR+I) = f(-x)RTR-f(x)TR+f(-x)R-f(x)I, \\
& &(TR-I)f(x)(R-I) = f(-x-1)T-f(-x-1)TR-f(x)R+f(x)I, \\
& &(TR+I)f(x)(R-I) = f(-x-1)T-f(-x-1)TR+f(x)R-f(x)I.
\end{eqnarray*}

\end{lemma}
The above formulas follow from the definitions of $R$, $T$, $I$ and elementary calculations, 
thus we may omit the proof. 

\par
Using the formulas in Lemma 2.1 we are able to derive the explicit expression of the operator $H^{(1)}$ 
from (\ref{eq:single Inter}). 
Denote the coefficients of $H^{(1)}$ by $\alpha^{(1)}(x)$, $\beta^{(1)}(x)$ and $\gamma^{(1)}(x)$:
$$
H^{(1)}=\alpha^{(1)}(x)(R-I)+\beta^{(1)}(x)(TR-I)+\gamma^{(1)}(x).
$$

\begin{prop}
The Dunkl shift operator $H^{(1)}$ which satisfies (\ref{eq:single Inter}) with $\mathcal{F}_{\phi}$ given by (\ref{eq:opF}) is 
\begin{equation}
\label{eq:opH1}
H^{(1)}=\frac{\tilde{\chi}(-x)}{\chi(x)}(R-I)+\frac{\alpha(-x-1)\beta(x)\chi(-x-1)}{\tilde{\chi}(x)}(TR-I)+
\frac{\tilde{\chi}(-x)-\tilde{\chi}(x)}{\chi(x)}.
\end{equation}
\end{prop}
\begin{proof}
We derive $H^{(1)}$ directly from (\ref{eq:single Inter}). 
According Lemma 2.1, the left-hand side of (\ref{eq:single Inter}) is
$$
\mathcal{F}_{\phi}\circ H=\frac{\beta(-x)}{\chi(x)}RTR
+\frac{\beta(x)\alpha(-x-1)}{\tilde{\chi}(x)}T
-\bigg(\frac{\beta(x)\alpha(-x-1)}{\tilde{\chi}(x)}+\frac{\beta(x)}{\chi(x)}\bigg)TR
$$
$$\hspace{-4mm}
+\bigg(\frac{\beta(-x-1)\alpha(x)}{\tilde{\chi}(x)}-\frac{\alpha(x)+\alpha(-x)+\beta(-x)}{\chi(x)}\bigg)R
$$
$$\hspace{-5mm}
-\bigg(\frac{\beta(-x-1)\alpha(x)}{\tilde{\chi}(x)}-\frac{\alpha(x)+\alpha(-x)+\beta(x)}{\chi(x)}\bigg)I,
$$
and the right-hand side of (\ref{eq:single Inter}) is
$$\hspace{2mm}
H^{(1)}\circ \mathcal{F}_{\phi}=\frac{\alpha^{(1)}(x)\beta(-x)}{\tilde{\chi}(-x)}RTR
+\frac{\beta^{(1)}(x)}{\chi(-x-1)}T
-\bigg(\frac{(\alpha^{(1)}(x)-\gamma^{(1)}(x))\beta(x)}{\tilde{\chi}(x)}+\frac{\beta^{(1)}(x)}{\chi(-x-1)}\bigg)TR
$$
$$\hspace{-12mm}
+\bigg(\frac{\alpha^{(1)}(x)\beta(x-1)}{\tilde{\chi}(-x)}+\frac{\gamma^{(1)}(x)-\beta^{(1)}(x)}{\chi(x)}\bigg)R
$$
$$\hspace{6mm}
+\bigg(\frac{(\gamma^{(1)}(x)-\alpha^{(1)}(x))\beta(-x-1)}{\tilde{\chi}(x)}+\frac{\beta^{(1)}(x)-\gamma^{(1)}(x)}{\chi(x)}\bigg)I.
$$
Then by comparing the coefficients of the operators $RTR$, $T$, $TR$, $R$ and $I$ one has
\begin{equation}
\label{eq:coeH1 ab}
\alpha^{(1)}(x)=\frac{\tilde{\chi}(-x)}{\chi(x)}, \hspace{2mm}
\beta^{(1)}(x)=\frac{\alpha(-x-1)\beta(x)\chi(-x-1)}{\tilde{\chi}(x)},
\end{equation}
\vspace{-2mm}
\begin{equation}
\label{eq:coeH1 c}
\gamma^{(1)}(x)=\frac{\tilde{\chi}(-x)-\tilde{\chi}(x)}{\chi(x)}
=-(\mu+\alpha(x)+\alpha(-x)+\beta(x)+\beta(-x-1)),
\end{equation}
where the second equation in (\ref{eq:coeH1 c}) follows from (\ref{eq:2.6}).
\end{proof}

\par
If we add the intertwining relation
$
H\circ\mathcal{B}_{\phi}=\mathcal{B}_{\phi}\circ H^{(1)}
$
where $\mathcal{B}_{\phi}$ is also a Dunkl shift operator, then it will be given by 
(we rearranged it for the convenience of checking (\ref{eq:factorization1}) and (\ref{eq:factorization2}))
\begin{equation}
\label{eq:opB}
\mathcal{B}_{\phi}=\alpha(x)(R-I)\tilde{\chi}(x)+(TR+I)\alpha(x)\beta(-x-1)\chi(x).
\end{equation}
Further calculations imply that if the condition (\ref{eq:condition b}) holds
then the products of $\mathcal{F}_{\phi}$ and $\mathcal{B}_{\phi}$ 
can be expressed in terms of $H$ and $H^{(1)}$, respectively: 
\begin{eqnarray}
\mathcal{B}_{\phi}\circ \mathcal{F}_{\phi}
\hspace{-2.4mm}&=&\hspace{-2.4mm}
(H-\mu)\circ(H+\beta(x)+\beta(-x-1)), \label{eq:factorization1} \\
\mathcal{F}_{\phi}\circ \mathcal{B}_{\phi}
\hspace{-2.4mm}&=&\hspace{-2.4mm}
(H^{(1)}-\mu)\circ(H^{(1)}+\beta(x)+\beta(-x-1)).
\label{eq:factorization2}
\end{eqnarray}
The above equations can be considered as a generalized DT of the Dunkl shift operator $H$, 
and $H^{(1)}$ is the 1-step Darboux transformed Dunkl shift operator.

\begin{rmk}
In order to make sure that this generalized DT can be applied on $H^{(1)}$ in the same manner as before, 
the coefficients $\alpha^{(1)}(x)$, $\beta^{(1)}(x)$ and $\gamma^{(1)}(x)$ should satisfy the same properties as 
those of $H$. 
Again, we denote the 2-step Darboux transformed Dunkl shift operator by $H^{(2)}$ with the form 
$$
H^{(2)}=\alpha^{(2)}(x)(R-I)+\beta^{(2)}(x)(TR-I)+\gamma^{(2)}(x),
$$
and similarly with (\ref{eq:opF}) we define the operator
$$
\mathcal{F}^{(2)}=\frac{1}{\chi^{(2)}(x)}(R-I)+\frac{1}{\tilde{\chi}^{(2)}(x)}(TR+I)\beta^{(1)}(-x-1),
$$
where 
$\chi^{(2)}(x)=(I-R)[\phi^{(2)}_2(x)]$, $\tilde{\chi}^{(2)}(x)=(I+TR)[\beta^{(1)}(-x-1)\phi^{(2)}_2(x)]$, 
$\phi^{(2)}_2(x)$ is the second-step seed solution which is an eigenfunction of $H^{(1)}$ with eigenvalue $\mu_2$: 
$\phi^{(2)}_2(x)=\mathcal{F}_{\phi}[\phi^{(1)}_2(x)]$ 
(we may denote $\mathcal{F}_{\phi}$ by $\mathcal{F}^{(1)}$ and the first-step seed solution $\phi(x)$ by $\phi^{(1)}_1(x)$ 
such that $\phi^{(2)}_2(x)=\mathcal{F}^{(1)}[\phi^{(1)}_2(x)]$, $\phi^{(1)}_2(x)\neq \phi^{(1)}_1(x)$). 
Then it again follows from the following intertwining relation 
$$
\mathcal{F}^{(2)} \circ H^{(1)} = H^{(2)} \circ \mathcal{F}^{(2)}
$$
that 
$$
\alpha^{(2)}(x)=\frac{\tilde{\chi}^{(2)}(-x)}{\chi^{(2)}(x)}, \hspace{2mm}
\beta^{(2)}(x)=\frac{\alpha^{(1)}(-x-1)\beta^{(1)}(x)\chi^{(2)}(-x-1)}{\tilde{\chi}^{(2)}(x)},
$$
and 
$$
\gamma^{(2)}(x)
=\frac{\tilde{\chi}^{(2)}(-x)-\tilde{\chi}^{(2)}(x)}{\chi^{(2)}(x)}+\gamma^{(1)}(-x-1)
=\frac{\tilde{\chi}^{(2)}(-x)-\tilde{\chi}^{(2)}(x)}{\chi^{(2)}(x)}+\gamma^{(1)}(-x)
$$
$$\hspace{-40.5mm}
=\frac{\tilde{\chi}^{(2)}(-x)-\tilde{\chi}^{(2)}(x)}{\chi^{(2)}(x)}+\gamma^{(1)}(x).
$$
The three expressions of $\gamma^{(2)}(x)$ imply
$
\gamma^{(1)}(-x-1)=\gamma^{(1)}(-x)=\gamma^{(1)}(x),
$
which further requires 
$
\alpha(x)+\alpha(-x)=\alpha(-x-1)+\alpha(x+1)
$
and (\ref{eq:condition b}) according to (\ref{eq:coeH1 c}). 
In the second and the third equation with respect to $\gamma^{(2)}(x)$ we have assumed that 
$\beta^{(1)}(x)$ satisfies the condition (\ref{eq:condition b}). 
Note that these three expressions did not appear in the calculation of $\gamma^{(1)}(x)$ since $\gamma^{(0)}(x)=0$ in $H$. 
Therefore, by repeating this procedure we can conclude as follow. 
\end{rmk}

\begin{lemma}
Denote the n-step Darboux transformed Dunkl shift operator by $H^{(n)}$ ($n=1,2,\ldots$) with the form 
\begin{equation}
\label{eq:n-DTop}
H^{(n)}=\alpha^{(n)}(x)(R-I)+\beta^{(n)}(x)(TR-I)+\gamma^{(n)}(x).
\end{equation}
If the condition 
\begin{equation}
\label{eq:condition a}
\alpha(x)+\alpha(-x)=\alpha(-x-1)+\alpha(x+1)
\end{equation}
and $(\ref{eq:condition b})$ are satisfied, 
then the generalized DT can be applied on each $H^{(n)}$ ($n=1,2,\ldots$) in the same manner as we did on $H$.
\end{lemma}

\begin{proof}
By repeating the procedure in Remark 2.3 one finds that the conditions 
\begin{equation}
\label{eq:condition1 for n-DT}
\alpha^{(n)}(x)+\alpha^{(n)}(-x)=\alpha^{(n)}(-x-1)+\alpha^{(n)}(x+1),
\end{equation}
\begin{equation}
\label{eq:condition2 for n-DT}
\beta^{(n)}(-x)+\beta^{(n)}(x-1)=\beta^{(n)}(x)+\beta^{(n)}(-x-1),
\end{equation}
\begin{equation}
\label{eq:condition3 for n-DT}
\gamma^{(n)}(-x-1)=\gamma^{(n)}(-x)=\gamma^{(n)}(x).
\end{equation}
are sufficient for deriving the coefficients $\alpha^{(n+1)}(x)$, $\beta^{(n+1)}(x)$, $\gamma^{(n+1)}(x)$.
On the other hand, from the calculations of $H^{(1)}$ and $H^{(2)}$ it is easy to conclude that for $n=1,2,\ldots$,
\begin{equation}
\label{eq:coen-DT a,b}
\alpha^{(n)}(x)=\frac{\tilde{\chi}^{(n)}(-x)}{\chi^{(n)}(x)}, \hspace{2mm}
\beta^{(n)}(x)=\frac{\alpha^{(n-1)}(-x-1)\beta^{(n-1)}(x)\chi^{(n)}(-x-1)}{\tilde{\chi}^{(n)}(x)},
\end{equation}
\begin{equation}
\label{eq:coen-DT c}
\gamma^{(n)}(x)=
-(\mu^{(n)}+\alpha^{(n-1)}(x)+\alpha^{(n-1)}(-x)+\beta^{(n-1)}(x)+\beta^{(n-1)}(-x-1))
\end{equation}
where the properties $\chi^{(n)}(-x)=-\chi^{(n)}(x)$, $\tilde{\chi}^{(n)}(-x-1)=\tilde{\chi}^{(n)}(x)$ always hold. 
Then inductively one can see that the conditions 
(\ref{eq:condition1 for n-DT}), (\ref{eq:condition2 for n-DT}), (\ref{eq:condition3 for n-DT}) 
hold for $n=1,2,\ldots$ if (\ref{eq:condition a}) and (\ref{eq:condition b}) are satisfied. 
\end{proof}

\par
Let us notice that the conditions (\ref{eq:condition a}) and (\ref{eq:condition b}) are already satisfied 
by the coefficients of the Bannai-Ito operator, 
thus a multiple-step DT can be performed on the Bannai-Ito operator smoothly. 
In the next subsection, we will derive the 1-step exceptional Bannai-Ito operator and the corresponding eigenfunctions. 
Some important expressions about the eigenfunctions of the $n$-step exceptional Bannai-Ito operator will also be given 
thereafter.

\subsection{Generalized Darboux transformation of the Bannai-Ito operator}
\par
If $H$ is a Bannai-Ito operator where $\alpha(x)$, $\beta(x)$ are defined by (\ref{eq:BIcoes}), 
then $H^{(1)}$ can be called the 1-step exceptional Bannai-Ito operator. 
In this case, it holds that
\begin{equation}
\label{eq:BIconsts}
\alpha(x)+\alpha(-x)=\rho_1+\rho_2:=\alpha, \quad
\beta(x)+\beta(-x-1)=-(r_1+r_2):=-\beta,
\end{equation}
hence  
\begin{equation}
\label{eq:}
\gamma^{(1)}(x)=-(\mu+\alpha(x)+\alpha(-x)+\beta(x)+\beta(-x-1))
=-(\mu+\alpha-\beta).
\end{equation}
Therefore, the 1-step exceptional Bannai-Ito operator becomes
\begin{equation}
\label{eq:XopBI}
H^{(1)}=\frac{\tilde{\chi}(-x)}{\chi(x)}(R-I)+\frac{\alpha(-x-1)\beta(x)\chi(-x-1)}{\tilde{\chi}(x)}(TR-I)
-(\mu+\alpha-\beta).
\end{equation}
For normalization purpose we may rewrite the operators $\mathcal{F}_{\phi}$, $\mathcal{B}_{\phi}$ 
and $H^{(1)}$ into
\begin{equation}
\label{normalizations}
\hat{\mathcal{F}}_{\phi}=r(x)\mathcal{F}_{\phi}, \quad
\hat{\mathcal{B}}_{\phi}=\mathcal{B}_{\phi}r^{-1}(x), \quad
\hat{H}^{(1)}=r(x)H^{(1)}r^{-1}(x),
\end{equation}
where $r(x)$ is a decoupling coefficient whose explicit expression will be given in Lemma 4.1.

\par
Adopting the notations of \cite{XOP} here we may call $\hat{\mathcal{F}}_{\phi}$ a ``dressing" 
operator and $\hat{\mathcal{B}}_{\phi}$ an ``undressing" operator in view of their acting as 
dressing and undressing of a superscript $^{(1)}$:
\vspace{-1mm}
\begin{empheq}[left={\psi^{(1)}(x):=\empheqlbrace}]{align}
\hat{\mathcal{F}}_{\phi}[\psi(x)], \hspace{15mm} & \text{if }\hat{\mathcal{F}}_{\phi}[\psi(x)]\neq 0 \\
\frac{\sigma(x)r(x)}{\tilde{\chi}(x)\chi(x)\alpha(x)\omega(x)}, \hspace{2mm} & \text{otherwise} 
\label{DTeigenfunction}
\end{empheq}
$$
\hat{\mathcal{B}}_{\phi}[\psi^{(1)}(x)]=(\nu-\mu)(\nu-\beta)\psi(x),
$$
where $\sigma(x)$ satisfies $\sigma(x+1)=\sigma(x)$ and $\sigma(-x)=-\sigma(x)$, 
for example, $\sigma(x)=\sin(2\pi x)$, 
and $\omega(x)$ is the weight function associated with the Bannai-Ito operator $H$.
It can be easily checked by using the results in Section 4 (Lemma 4.7) that 
\begin{equation}
\label{eq:2.29}
\hat{H}^{(1)}\bigg[\frac{\sigma(x)r(x)}{\tilde{\chi}(x)\chi(x)\alpha(x)\omega(x)}\bigg]=
\mu \frac{\sigma(x)r(x)}{\tilde{\chi}(x)\chi(x)\alpha(x)\omega(x)},
\end{equation}
thus $\psi^{(1)}(x)$ is the eigenfunction of the normalized 1-step exceptional Bannai-Ito operator $\hat{H}^{(1)}$.
Therefore, the DT 
$$
(H,\{\psi(x)\})\mapsto (\hat{H}^{(1)},\{\psi^{(1)}(x)\})
$$
is an isospectral transformation:
$$
H[\psi(x)]=\nu\psi(x), \quad
\hat{H}^{(1)}[\psi^{(1)}(x)]=\nu\psi^{(1)}(x).
$$
In fact, the equation (\ref{eq:2.29}) is equivalent with 
\begin{equation}
\label{eq:2.30}
H^{(1)}\bigg[\frac{\sigma(x)}{\tilde{\chi}(x)\chi(x)\alpha(x)\omega(x)}\bigg]=
\mu \frac{\sigma(x)}{\tilde{\chi}(x)\chi(x)\alpha(x)\omega(x)}
\end{equation}
in view of (\ref{normalizations}). 
Here we briefly check (\ref{eq:2.30}) using the expression (\ref{eq:opH1}).
$$
H^{(1)}\bigg[\frac{\sigma(x)}{\tilde{\chi}(x)\chi(x)\alpha(x)\omega(x)}\bigg]=
\frac{\tilde{\chi}(-x)}{\chi(x)}
\bigg(\frac{1}{\tilde{\chi}(-x)}-\frac{1}{\tilde{\chi}(x)}\bigg)
\frac{\sigma(x)}{\chi(x)\alpha(x)\omega(x)}
+\frac{\tilde{\chi}(-x)-\tilde{\chi}(x)}{\chi(x)}
\frac{\sigma(x)}{\tilde{\chi}(x)\chi(x)\alpha(x)\omega(x)}
$$
$$\hspace{31mm}
+\frac{\alpha(-x-1)\beta(x)\chi(-x-1)}{\tilde{\chi}(x)}
\bigg(\frac{-\beta(-x-1)}{\chi(-x-1)\alpha(-x-1)}-\frac{\beta(x)}{\chi(x)\alpha(x)}\bigg)
\frac{\sigma(x)}{\tilde{\chi}(x)\beta(x)\omega(x)}
$$
$$\hspace{9mm}
=\frac{\sigma(x)}{\tilde{\chi}(x)\omega(x)}
\bigg(\frac{-\beta(-x-1)}{\tilde{\chi}(x)}-\frac{\beta(x)\alpha(-x-1)\chi(-x-1)}{\tilde{\chi}(x)\chi(x)\alpha(x)}\bigg)
$$
$$\hspace{23mm}
=\frac{\sigma(x)}{\tilde{\chi}(x)\omega(x)}
\bigg(-\frac{\beta(-x-1)\alpha(x)\chi(x)+\beta(x)\alpha(-x-1)\chi(-x-1)}{\tilde{\chi}(x)\chi(x)\alpha(x)}\bigg).
$$
In the first equation, the properties $\chi(-x)=-\chi(x)$, $\tilde{\chi}(-x-1)=\tilde{\chi}(x)$, $\sigma(-x)=-\sigma(x)$, 
$\sigma(-x-1)=\sigma(-x)=-\sigma(x)$ and (\ref{eq:WC1}), (\ref{eq:WC2}) in Lemma 4.7 have been used. 
Then the first 2 terms in the right-hand side of the first equation annihilate immediately, 
the third term can be simplified as the right-hand side of the second equation. 
Finally, with the help of (\ref{eq:2.5}) we arrive at (\ref{eq:2.30}).

\par
In the same way it turns out that (see (\ref{eq:opB}))
\begin{equation}
\mathcal{B}_{\phi}\bigg[\frac{\sigma(x)}{\tilde{\chi}(x)\chi(x)\alpha(x)\omega(x)}\bigg]=0,
\end{equation}
since
$$
(R-I)\bigg[\frac{\sigma(x)}{\chi(x)\alpha(x)\omega(x)}\bigg]=0, \hspace{2mm}
(TR+I)\bigg[\frac{\beta(-x-1)\sigma(x)}{\tilde{\chi}(x)\omega(x)}\bigg]=0
$$
hold for the same properties as we listed in the previous paragraph.

\par
To summarize the above results we give the following theorem which serves as the formulation of the 
generalized Darboux transformation for the Bannai-Ito operator.
\begin{theorem}
\label{thm:gDT}
The 1-step exceptional Bannai-Ito operator $H^{(1)}$ 
and the Bannai-Ito operator $H$ satisfy 
the following intertwining relations:
\begin{equation}
\label{eq:gDT}
\mathcal{F}_{\phi}\circ H=H^{(1)}\circ \mathcal{F}_{\phi}, \quad
H\circ\mathcal{B}_{\phi}=\mathcal{B}_{\phi}\circ H^{(1)}
\end{equation}
where $\mathcal{F}_{\phi}$, $\mathcal{B}_{\phi}$ are both Dunkl shift operators (see (\ref{eq:opF}), (\ref{eq:opB})). 
The operator $\mathcal{F}_{\phi}$ satisfies the condition $\mathcal{F}_{\phi}[\phi(x)]=0$, 
where $\phi(x)$ is an eigenfunction of $H$ called a seed solution.
Moreover, the products of $\mathcal{F}_{\phi}$ and $\mathcal{B}_{\phi}$ can be expressed in terms of $H$ and $H^{(1)}$:
\begin{eqnarray}
\mathcal{B}_{\phi}\circ \mathcal{F}_{\phi}
\hspace{-2.4mm}&=&\hspace{-2.4mm}
(H-\mu)\circ(H-\beta), \label{eq:BIfactorization1} \\
\mathcal{F}_{\phi}\circ \mathcal{B}_{\phi}
\hspace{-2.4mm}&=&\hspace{-2.4mm}
(H^{(1)}-\mu)\circ(H^{(1)}-\beta).
\label{eq:BIfactorization2}
\end{eqnarray}
The relations (\ref{eq:BIfactorization1}), (\ref{eq:BIfactorization2}) follow from 
(\ref{eq:factorization1}), (\ref{eq:factorization2}) and (\ref{eq:BIconsts}).
\end{theorem}

\par
Later in Section 4 we will show that with a well selected seed solution $\phi(x)$ 
and the related decoupling coefficient $r(x)$, 
the 1-step Darboux transformed eigenfunctions $\{\hat{\mathcal{F}}_{\phi}[B_n(x)]\}$ 
($B_n(x)$ are the Bannai-Ito polynomials)  
possess the ``exceptional'' feature (gaps in their degree sequences) and are orthogonal with respect to 
a discrete measure on the exceptional Bannai-Ito grid. 

\par
As an immediate result of the intertwining relations (\ref{eq:gDT}), we have
$$
H\circ(\mathcal{B}_{\phi}\circ\mathcal{F}_{\phi})=(\mathcal{B}_{\phi}\circ\mathcal{F}_{\phi})\circ H,
$$
$$
(\mathcal{F}_{\phi}\circ\mathcal{B}_{\phi})\circ H^{(1)}=H^{(1)}\circ(\mathcal{F}_{\phi}\circ\mathcal{B}_{\phi}),
$$
which indicate that $H$ and $H^{(1)}$ are commutable with $(\mathcal{B}_{\phi}\circ\mathcal{F}_{\phi})$ 
and $(\mathcal{F}_{\phi}\circ\mathcal{B}_{\phi})$, respectively.
In a more generic setting, if we denote $H$ by $H^{(0)}$, and the exceptional operator obtained after $n$ 
steps of DT by $H^{(n)}$ as we did before, 
while the same notations adopting to the operators $\mathcal{F}^{(n)}$ and $\mathcal{B}^{(n)}$ with 
$\mathcal{F}^{(1)}=\mathcal{F}_{\phi}$, $\mathcal{B}^{(1)}=\mathcal{B}_{\phi}$, 
then we can give the intertwining relations with respect to the multiple-step Darboux transformed Bannai-Ito operators. 
\begin{corollary}
\label{thm:n-Inter}
The exceptional Bannai-Ito operator $H^{(n+1)}$ and the Bannai-Ito operator $H^{(0)}$ satisfy the following 
intertwining relations for $n=1,2,\ldots$,
\begin{equation}
\label{n-intertwining1}
(\mathcal{F}^{(n)}\circ\cdots\circ\mathcal{F}^{(1)})\circ H^{(0)}
=H^{(n+1)}\circ(\mathcal{F}^{(n)}\circ\cdots\circ\mathcal{F}^{(1)}),
\end{equation}
\begin{equation}
\label{n-intertwining2}
H^{(0)}\circ(\mathcal{B}^{(1)}\circ\cdots\circ\mathcal{B}^{(n)})
=(\mathcal{B}^{(1)}\circ\cdots\circ\mathcal{B}^{(n)})\circ H^{(n+1)}.
\end{equation}
\end{corollary}
\begin{proof}
This corollary follows inductively from the construction of the multiple-step exceptional Bannai-Ito operators 
$H^{(n)}$ ($n=1,2,\ldots$). 
According to Theorem \ref{thm:gDT} we know that the exceptional Bannai-Ito operators $H^{(1)}$, $H^{(2)}$, $H^{(3)}$, $\ldots$ 
can be obtained through the following intertwining relations
\begin{eqnarray*}
\mathcal{F}^{(1)}\circ H^{(0)} = H^{(1)}\circ \mathcal{F}^{(1)}, \hspace{-2.4mm}& &\hspace{-2.4mm}
H^{(0)}\circ\mathcal{B}^{(1)}=\mathcal{B}^{(1)}\circ H^{(1)}, \\
\mathcal{F}^{(2)}\circ H^{(1)} = H^{(2)}\circ \mathcal{F}^{(2)}, \hspace{-2.4mm}& &\hspace{-2.4mm} 
H^{(1)}\circ\mathcal{B}^{(2)}=\mathcal{B}^{(2)}\circ H^{(2)}, \\
\mathcal{F}^{(3)}\circ H^{(2)} = H^{(3)}\circ \mathcal{F}^{(3)}, \hspace{-2.4mm}& &\hspace{-2.4mm} 
H^{(2)}\circ\mathcal{B}^{(3)}=\mathcal{B}^{(3)}\circ H^{(3)}, \\
\hspace{-2.4mm}& \vdots &\hspace{-2.4mm}
\end{eqnarray*}
By applying $\mathcal{F}^{(2)}$ on the left-hand side of the first equation 
and then rewriting the result using the third equation we have
$$
\mathcal{F}^{(2)}\circ \mathcal{F}^{(1)}\circ H^{(0)} = \mathcal{F}^{(2)}\circ H^{(1)}\circ \mathcal{F}^{(1)}
=H^{(2)}\circ \mathcal{F}^{(2)}\circ \mathcal{F}^{(1)}.
$$
Similarly, by applying $\mathcal{B}^{(2)}$ on the right-hand side of the second equation 
and then rewriting the result using the fourth equation we have 
$$
H^{(0)}\circ\mathcal{B}^{(1)}\circ\mathcal{B}^{(2)}=\mathcal{B}^{(1)}\circ H^{(1)}\circ\mathcal{B}^{(2)}
=\mathcal{B}^{(1)}\circ \mathcal{B}^{(2)}\circ H^{(2)}.
$$
Repeating this procedure finally one will arrive at (\ref{n-intertwining1}) and (\ref{n-intertwining2}).
\end{proof}

\par
Again, we can deduce
$$
H^{(0)}\circ(\mathcal{B}^{(1)}\circ\cdots\circ\mathcal{B}^{(n)})\circ(\mathcal{F}^{(n)}\circ\cdots\circ\mathcal{F}^{(1)})=
(\mathcal{B}^{(1)}\circ\cdots\circ\mathcal{B}^{(n)})\circ(\mathcal{F}^{(n)}\circ\cdots\circ\mathcal{F}^{(1)})\circ H^{(0)},
$$
$$
(\mathcal{F}^{(n)}\circ\cdots\circ\mathcal{F}^{(1)})\circ(\mathcal{B}^{(1)}\circ\cdots\circ\mathcal{B}^{(n)})\circ H^{(n+1)}=
H^{(n+1)}\circ(\mathcal{F}^{(n)}\circ\cdots\circ\mathcal{F}^{(1)})\circ(\mathcal{B}^{(1)}\circ\cdots\circ\mathcal{B}^{(n)}),
$$
which indicate that $H$ and $H^{(n+1)}$ are commutable with 
$(\mathcal{B}^{(1)}\circ\cdots\circ\mathcal{B}^{(n)})\circ(\mathcal{F}^{(n)}\circ\cdots\circ\mathcal{F}^{(1)})$ 
and 
$(\mathcal{F}^{(n)}\circ\cdots\circ\mathcal{F}^{(1)})\circ(\mathcal{B}^{(1)}\circ\cdots\circ\mathcal{B}^{(n)})$, respectively.

\subsection{Determinant expression of multiple-step exceptional eigenfunctions}
\par
From the intertwining relations in Corollary \ref{thm:n-Inter} one knows that the eigenfunctions of the $n$-step 
exceptional Bannai-Ito operator $H^{(n)}$ can be given by 
\begin{equation}
\label{eq:n-step eigenfunction}
\phi_{m}^{(n)}(x)= \mathcal{F}^{(n-1)}\circ\cdots\circ\mathcal{F}^{(1)}[\phi^{(1)}_{m}(x)] \quad
(n\geq 1),
\end{equation}
where $\phi^{(1)}_{m}(x)$ is an eigenfunction of $H^{(0)}$ (or $H$, the original Bannai-Ito operator)
with eigenvalue $\mu_m$. 
Please do not mistake $\phi_{m}^{(k)}(x)$ and $\phi_{n}^{(k)}(x)$ by the 
eigenfunctions of order $m$ and order $n$, 
here we use the subscripts $m$, $n$ only to indicate they are different.
Recall that
\begin{equation}
\label{eq:n-step opF}
\mathcal{F}^{(n)}=\frac{1}{\chi^{(n)}(x)}(R-I)+\frac{1}{\tilde{\chi}^{(n)}(x)}(TR+I)\beta^{(n-1)}(-x-1),
\end{equation}
where $\beta^{(n)}(x)$ is given by (\ref{eq:coen-DT a,b}), and
\begin{equation}
\label{eq:n-step opX}
\hspace{-15mm}
\chi^{(n)}(x)
:=\chi^{(n)}_{n}(x)
=(I-R)[\phi^{(n)}_{n}(x)],
\end{equation}
\begin{equation}
\label{eq:n-step opXX}
\hspace{6.77mm}
\tilde{\chi}^{(n)}(x)
:=\tilde{\chi}^{(n)}_{n}(x)
=(I+TR)[\beta^{(n-1)}(-x-1)\phi^{(n)}_{n}(x)].
\end{equation}
Namely, the left-hand sides of (\ref{def: X}) and (\ref{def: XX}) are represented by $\chi^{(1)}(x)$ and $\tilde{\chi}^{(1)}(x)$, 
$\phi^{(n)}_{n}(x)$ is the $n$th-step seed solution: 
\begin{equation}
\label{eq:n-step seed}
H^{(n-1)}[\phi^{(n)}_{n}(x)]=\mu_n\phi^{(n)}_{n}(x), \hspace{2mm}
\phi^{(n)}_{n}(x)=\mathcal{F}^{(n-1)}[\phi^{(n-1)}_n(x)].
\end{equation}
For the simplicity of calculation we introduce the following notations 
(whose special cases are (\ref{eq:n-step opX}) and (\ref{eq:n-step opXX}))
\begin{equation}
\label{eq:n-step opX, opXX}
\chi^{(n)}_{m}(x)
=(I-R)[\phi^{(n)}_{m}(x)],
\hspace{2mm}
\tilde{\chi}^{(n)}_{m}(x)
=(I+TR)[\beta^{(n-1)}(-x-1)\phi^{(n)}_{m}(x)].
\end{equation}
It then follows from the definition that, for $n\geq 2$ we have
\begin{equation}
\label{eq:n-step opX formula}
\chi^{(n)}_{m}(x)
=\frac{\tilde{\chi}^{(n-1)}_{m}(x)}{\tilde{\chi}^{(n-1)}(x)}-\frac{\tilde{\chi}^{(n-1)}_{m}(-x)}{\tilde{\chi}^{(n-1)}(-x)},
\end{equation}
and
\begin{equation}
\label{eq:n-step opXX formula}
\tilde{\chi}^{(n)}_{m}(x)
=(\mu_{m}-\mu_{n-1})\frac{\tilde{\chi}^{(n-1)}_{m}(x)}{\tilde{\chi}^{(n-1)}(x)},
\end{equation}
where the expression 
\begin{equation}
\label{eq:n-step phi}
\phi_{m}^{(n)}(x)= \mathcal{F}^{(n-1)}[\phi^{(n-1)}_{m}(x)]
=\frac{\tilde{\chi}^{(n-1)}_{m}(x)}{\tilde{\chi}^{(n-1)}(x)}-\frac{\chi^{(n-1)}_{m}(x)}{\chi^{(n-1)}(x)}
\end{equation}
and the eigenvalue equation of $\phi_{m}^{(n-1)}(x)$ (specifically, the $(n-1)$-step version of (\ref{eq:2.5})) 
have been used for deriving (\ref{eq:n-step opX formula}) and (\ref{eq:n-step opXX formula}). 

\par
As it has been shown in \cite{ECH, EML, EHJ, EH, MDBT, MIOP}, 
multiple-step exceptional orthogonal polynomials can be expressed in Wronskian determinants 
whose entries are the classical orthogonal polynomials and their derivatives. 
Similarly, in the exceptional Bannai-Ito case, 
for example, the eigenfunctions of the 1-step exceptional Bannai-Ito operator $H^{(1)}$ are the following determinant. 
$$
\phi^{(2)}_{m}(x)=
\frac{1}{\tilde{\chi}^{(1)}(x)\chi^{(1)}(x)}
\begin{vmatrix}
\tilde{\chi}^{(1)}_{m}(x) & \chi^{(1)}_{m}(x) \\
\tilde{\chi}^{(1)}(x) & \chi^{(1)}(x)
\end{vmatrix}
$$
For the eigenfunctions of the $n$-step exceptional Bannai-Ito operator $H^{(n)}$ ($n\geq 2$), 
we show in the next theorem that they can always be expressed in a $3\times 3$ determinant.
\begin{theorem}
\label{thm:n-DT}
The eigenfunctions of the $n$-th step exceptional Bannai-Ito operator $H^{(n)}$ ($n\geq 2$) can be expressed as 
the $3\times 3$ determinant 
\begin{equation}
\label{eq:n-step eigenfunction}
\phi_{m}^{(n+1)}(x)
=\prod^{n-2}_{j=1}\frac{(\mu_{m}-\mu_{j})}{(\mu_{n}-\mu_{j})}
\begin{vmatrix}
\mu_{m}\tilde{\chi}^{(1)}_{m}(x) & \tilde{\chi}^{(1)}_{m}(x)  & \tilde{\chi}^{(1)}_{m}(-x) \\
\mu_{n}\tilde{\chi}^{(1)}_{n}(x)   & \tilde{\chi}^{(1)}_{n}(x)   & \tilde{\chi}^{(1)}_{n}(-x)  \\
\mu_{n-1}\tilde{\chi}^{(1)}_{n-1}(x)  & \tilde{\chi}^{(1)}_{n-1}(x) & \tilde{\chi}^{(1)}_{n-1}(-x) 
\end{vmatrix}
\end{equation}
$$\hspace{10mm}
\cdot
\bigg[(\mu_{n}-\mu_{n-1})\tilde{\chi}^{(1)}_{n}(x)
\begin{vmatrix}
\tilde{\chi}^{(1)}_{n}(x)   & \tilde{\chi}^{(1)}_{n}(-x)  \\
\tilde{\chi}^{(1)}_{n-1}(x) & \tilde{\chi}^{(1)}_{n-1}(-x)
\end{vmatrix}
\bigg]^{-1}
$$
where $\mu_n$ is the eigenvalue of $\phi^{(1)}_{n}(x): H[\phi^{(1)}_{n}(x)]=\mu_n\phi^{(1)}_{n}(x)$, and 
$$
\tilde{\chi}^{(1)}_{n}(x)=\beta(-x-1)\phi^{(1)}_{n}(x)+\beta(x)\phi^{(1)}_{n}(-x-1).
$$

\end{theorem}

\begin{proof}
First, it is easily seen from (\ref{eq:n-step opX formula}) and (\ref{eq:n-step opXX formula}) that, for $n \geq 2$ we have 
$$
\tilde{\chi}^{(n)}_{m}(x)
=\frac{\prod\limits^{n-1}_{j=1}(\mu_{m}-\mu_{j})}{\prod\limits^{n-2}_{k=1}(\mu_{n-1}-\mu_{k})}
\frac{\tilde{\chi}^{(1)}_{m}(x)}{\tilde{\chi}^{(1)}_{n-1}(x)}, \quad
\chi^{(n)}_{m}(x)
=\frac{\prod\limits^{n-2}_{j=1}(\mu_{m}-\mu_{j})}{\prod\limits^{n-2}_{k=1}(\mu_{n-1}-\mu_{k})}
\bigg(\frac{\tilde{\chi}^{(1)}_{m}(x)}{\tilde{\chi}^{(1)}_{n-1}(x)}-\frac{\tilde{\chi}^{(1)}_{m}(-x)}{\tilde{\chi}^{(1)}_{n-1}(-x)}\bigg).
$$
Then from (\ref{eq:n-step phi}) it follows that 
$$\hspace{-66mm}
\phi_m^{(n+1)}(x) 
=\frac{\tilde{\chi}^{(n)}_{m}(x)}{\tilde{\chi}^{(n)}(x)}-\frac{\chi^{(n)}_{m}(x)}{\chi^{(n)}(x)}
$$
$$\hspace{24mm}
=\prod^{n-1}_{j=1}\frac{(\mu_{m}-\mu_{j})}{(\mu_{n}-\mu_{j})}
\frac{\tilde{\chi}^{(1)}_{m}(x)}{\tilde{\chi}^{(1)}_{n}(x)}
-\prod^{n-2}_{j=1}\frac{(\mu_{m}-\mu_{j})}{(\mu_{n}-\mu_{j})}
\frac{\tilde{\chi}^{(1)}_{m}(x)\tilde{\chi}^{(1)}_{n-1}(-x)-\tilde{\chi}^{(1)}_{m}(-x)\tilde{\chi}^{(1)}_{n-1}(x)}
{\tilde{\chi}^{(1)}_{n}(x)\tilde{\chi}^{(1)}_{n-1}(-x)-\tilde{\chi}^{(1)}_{n}(-x)\tilde{\chi}^{(1)}_{n-1}(x)}
$$
$$\hspace{24mm}
=\prod^{n-2}_{j=1}\frac{(\mu_{m}-\mu_{j})}{(\mu_{n}-\mu_{j})}
\bigg[
\frac{(\mu_{m}-\mu_{n-1})\tilde{\chi}^{(1)}_{m}(x)}{(\mu_{n}-\mu_{n-1})\tilde{\chi}^{(1)}_{n}(x)}
-\frac{\tilde{\chi}^{(1)}_{m}(x)\tilde{\chi}^{(1)}_{n-1}(-x)-\tilde{\chi}^{(1)}_{m}(-x)\tilde{\chi}^{(1)}_{n-1}(x)}
{\tilde{\chi}^{(1)}_{n}(x)\tilde{\chi}^{(1)}_{n-1}(-x)-\tilde{\chi}^{(1)}_{n}(-x)\tilde{\chi}^{(1)}_{n-1}(x)}
\bigg].
$$
Finally, the formula in the square brackets can be rewritten into 
 $$
\begin{vmatrix}
\mu_{m}\tilde{\chi}^{(1)}_{m}(x) & \tilde{\chi}^{(1)}_{m}(x)  & \tilde{\chi}^{(1)}_{m}(-x) \\
\mu_{n}\tilde{\chi}^{(1)}_{n}(x)   & \tilde{\chi}^{(1)}_{n}(x)   & \tilde{\chi}^{(1)}_{n}(-x)  \\
\mu_{n-1}\tilde{\chi}^{(1)}_{n-1}(x)  & \tilde{\chi}^{(1)}_{n-1}(x) & \tilde{\chi}^{(1)}_{n-1}(-x) 
\end{vmatrix}\cdot
\bigg[(\mu_{n}-\mu_{n-1})\tilde{\chi}^{(1)}_{n}(x)
\begin{vmatrix}
\tilde{\chi}^{(1)}_{n}(x)   & \tilde{\chi}^{(1)}_{n}(-x)  \\
\tilde{\chi}^{(1)}_{n-1}(x) & \tilde{\chi}^{(1)}_{n-1}(-x)
\end{vmatrix}
\bigg]^{-1}.
$$
\end{proof}

\par
Let us observe the right-hand side of (\ref{eq:n-step eigenfunction}), 
$\phi^{(n+1)}_m(x)$ vanishes if $m=i$ for $i=1,2,\ldots, n$. 
The cases $m=n$ and $m=n-1$ are due to the $3\times 3$ determinant part, 
while the cases $m\in\{1,2,\ldots, n-2\}$ are due to the prefactor $\prod^{n-2}_{j=1}(\mu_{m}-\mu_{j})$. 
In this way, there are in total $n$ eigenfunctions been deleted from the eigenfunction sequence of 
the $n$-step exceptional Bannai-Ito operator $H^{(n)}$.

\section{Quasi-polynomial eigenfunctions of Bannai-Ito operator}
\par
In this section we consider a special class of eigenfunctions of the Bannai-Ito operator, 
which are called quasi-polynomial eigenfunctions. 
A quasi-polynomial eigenfunction is the product of a gauge factor and a polynomial part:
$$
H[\xi(x)p(x)]=\lambda\xi(x)p(x),
$$
where $\xi(x)$ is a function in $x$ and $p(x)$ is a polynomial in $x$, $\lambda$ is the corresponding eigenvalue. The Bannai-Ito operator may have several sequences of quasi-polynomial eigenfunctions. 
From these quasi-polynomial eigenfunctions, in the next section, we will choose the seed solutions of the DT. 
This plays an important role in the construction of exceptional Bannai-Ito polynomials.

\par
From the definition of the quasi-polynomial eigenfunction, we will derive all possible gauge factors $\xi(x)$.
First, let us consider the conjugated operator $\tilde{H}=\xi^{-1}H\xi$:
\begin{eqnarray*}
\tilde{H}
\hspace{-2.4mm}& = &\hspace{-2.4mm}
\frac{1}{\xi(x)}\bigg[\alpha(x)(R-I)\xi(x)+\beta(x)(TR-I)\xi(x)\bigg] \\
\hspace{-2.4mm}& = &\hspace{-2.4mm}
\alpha(x)\frac{\xi(-x)}{\xi(x)}(R-I)+
\beta(x)\frac{\xi(-x-1)}{\xi(x)}(TR-I)+\alpha(x)\bigg[\frac{\xi(-x)}{\xi(x)}-1\bigg]+\beta(x)\bigg[\frac{\xi(-x-1)}{\xi(x)}-1\bigg].
\end{eqnarray*}
The operator $\tilde{H}$ has polynomial eigenfunctions if and only if 
$\tilde{H}$ is also a Bannai-Ito operator, i.e. there exist real numbers $\rho'_1$, $\rho'_2$, $r'_1$, $r'_2$ 
and $\rho''_1$, $\rho''_2$, $r''_1$, $r''_2$
such that both $(3.1)$ and one of $(3.2)$, $(3.3)$ are satisfied
\begin{equation}
\alpha(x)\bigg[\frac{\xi(-x)}{\xi(x)}-1\bigg]+\beta(x)\bigg[\frac{\xi(-x-1)}{\xi(x)}-1\bigg] 
=const.
\end{equation}
\begin{equation}
\alpha(x)\frac{\xi(-x)}{\xi(x)}=\frac{(x-\rho'_1)(x-\rho'_2)}{-2x}, \hspace{2mm}
\beta(x)\frac{\xi(-x-1)}{\xi(x)}=\frac{(x-r'_1+\frac{1}{2})(x-r'_2+\frac{1}{2})}{2x+1},
\end{equation}
\begin{equation}
\alpha(x)\frac{\xi(-x)}{\xi(x)}=\frac{(x-\rho''_1)(x-\rho''_2)}{2x}, \hspace{2mm}
\beta(x)\frac{\xi(-x-1)}{\xi(x)}=-\frac{(x-r''_1+\frac{1}{2})(x-r''_2+\frac{1}{2})}{2x+1}.
\end{equation}

\par
Further calculations lead to the following lemma.
\begin{lemma}
Let $H=\alpha(x)(R-I)+\beta(x)(TR-I)$ be a Bannai-Ito operator, then the conjugated operator 
$\tilde{H}=\xi^{-1}H\xi$ has polynomial eigenfunctions if and only if 
$\xi(x)$ satisfies
\begin{equation}
\frac{\xi(-x)}{\xi(x)}=\frac{(x-\rho'_1)(x-\rho'_2)}{(x-\rho_1)(x-\rho_2)}, \quad
\frac{\xi(-x-1)}{\xi(x)}=\frac{(x-r'_1+\frac{1}{2})(x-r'_2+\frac{1}{2})}{(x-r_1+\frac{1}{2})(x-r_2+\frac{1}{2})},
\end{equation}
where $\rho'_1\rho'_2=\rho_1\rho_2$ and $r'_1r'_2=r_1r_2$; 
or $\xi(x)$ satisfies
\begin{equation}
\frac{\xi(-x)}{\xi(x)}=-\frac{(x-\rho''_1)(x-\rho''_2)}{(x-\rho_1)(x-\rho_2)}, \quad
\frac{\xi(-x-1)}{\xi(x)}=-\frac{(x-r''_1+\frac{1}{2})(x-r''_2+\frac{1}{2})}{(x-r_1+\frac{1}{2})(x-r_2+\frac{1}{2})},
\end{equation}
where $\rho''_1\rho''_2=-\rho_1\rho_2$ and $r''_1r''_2=-r_1r_2$.
\end{lemma}

\begin{proof}
It is easy to see that $(3.4)$ and $(3.5)$ follow from $(3.2)$ and $(3.3)$, respectively. 
Thus, $(3.1)$ can be rewritten into 
\begin{eqnarray*}
const.\hspace{-2.4mm}&=&\hspace{-2.4mm}
\frac{(x-\rho'_1)(x-\rho'_2)-(x-\rho_1)(x-\rho_2)}{-2x}
+\frac{(x-r'_1+\frac{1}{2})(x-r'_2+\frac{1}{2})-(x-r_1+\frac{1}{2})(x-r_2+\frac{1}{2})}{2x+1} \\
\hspace{-2.4mm}&=&\hspace{-2.4mm}
-\frac{\rho_1+\rho_2-\rho'_1-\rho'_2}{2}+\frac{r_1+r_2-r'_1-r'_2}{2}-\frac{\rho'_1\rho'_2-\rho_1\rho_2}{2x}
+\frac{r'_1r'_2-r_1r_2}{2x+1}
\end{eqnarray*}
which implies $\rho'_1\rho'_2-\rho_1\rho_2=0$ and $r'_1r'_2-r_1r_2=0$; 
or
\begin{eqnarray*}
const.\hspace{-2.4mm}&=&\hspace{-2.4mm}
\frac{(x-\rho''_1)(x-\rho''_2)+(x-\rho_1)(x-\rho_2)}{2x}
-\frac{(x-r''_1+\frac{1}{2})(x-r''_2+\frac{1}{2})+(x-r_1+\frac{1}{2})(x-r_2+\frac{1}{2})}{2x+1} \\
\hspace{-2.4mm}&=&\hspace{-2.4mm}
-\frac{\rho_1+\rho_2+\rho''_1+\rho''_2}{2}+\frac{r_1+r_2+r''_1+r''_2}{2}-\frac{1}{2}
+\frac{\rho''_1\rho''_2+\rho_1\rho_2}{2x}-\frac{r''_1r''_2+r_1r_2}{2x+1}
\end{eqnarray*}
which implies $\rho''_1\rho''_2+\rho_1\rho_2=0$ and $r''_1r''_2+r_1r_2=0$.
\end{proof}
It then follows that there are both $4$ sets of parameterizations ($\sigma_1$, $\sigma_2$, $\sigma_3$, $\sigma_4$) 
for $\rho'_1$, $\rho'_2$, $r'_1$, $r'_2$ 
and $\rho''_1$, $\rho''_2$, $r''_1$, $r''_2$ ($\sigma_5$, $\sigma_6$, $\sigma_7$, $\sigma_8$): 
$$
\sigma_1=\{\rho_1,\rho_2,r_1,r_2\}, \hspace{2mm}
\sigma_2=\{-\rho_1,-\rho_2,-r_1,-r_2\},
$$
$$
\sigma_3=\{-\rho_1,-\rho_2,r_1,r_2\}, \hspace{2mm}
\sigma_4=\{\rho_1,\rho_2,-r_1,-r_2\},
$$ 
$$
\sigma_5=\{-\rho_1,\rho_2,-r_1,r_2\}, \hspace{2mm}
\sigma_6=\{\rho_1,-\rho_2,r_1,-r_2\},
$$
$$
\sigma_7=\{\rho_1,-\rho_2,-r_1,r_2\}, \hspace{2mm}
\sigma_8=\{-\rho_1,\rho_2,r_1,-r_2\}.
$$

\begin{rmk}
One may ask why there are only 8 sets of parameterizations, for example, the condition $\rho'_1\rho'_2=\rho_1\rho_2$ 
leads to $\rho'_1=k\rho_1$ and $\rho'_2=\rho_2/k$ where $k$ can be any real number. 
We will show that the conditions (3.4) and (3.5) imply that only the choices $k=\pm 1$ are allowed. 
In fact, if we let $\rho'_1=k\rho_1$ and $\rho'_2=\rho_2/k$, then the first equation of (3.4) can be rewritten into 
\begin{equation}
\frac{\xi(-x)}{\xi(x)}=\frac{(x-k\rho_1)(x-\rho_2/k)}{(x-\rho_1)(x-\rho_2)},
\end{equation}
which becomes the next equation with $x$ replaced by $-x$
\begin{equation}
\frac{\xi(x)}{\xi(-x)}=\frac{(x+k\rho_1)(x+\rho_2/k)}{(x+\rho_1)(x+\rho_2)}.
\end{equation}
Then we have
$$
1=\frac{\xi(-x)}{\xi(x)}\cdot\frac{\xi(x)}{\xi(-x)}=
\frac{(x-k\rho_1)(x-\rho_2/k)}{(x-\rho_1)(x-\rho_2)}\cdot\frac{(x+k\rho_1)(x+\rho_2/k)}{(x+\rho_1)(x+\rho_2)}
=\frac{(x^2-k^2\rho^2_1)(x^2-\rho^2_2/k^{2})}{(x^2-\rho^2_1)(x^2-\rho^2_2)}.
$$
The above equation have 4 solutions: $k^2=1$ and $k^2=\rho^2_2/\rho^2_1$. 
However, the former two solutions lead to $(\rho'_1,\rho'_2)=(\rho_1, \rho_2)$ or $(-\rho_1,-\rho_2)$ 
while the latter two solutions lead to $(\rho'_1,\rho'_2)=(\rho_2, \rho_1)$ or $(-\rho_2,-\rho_1)$. 
One should notice that nothing changes if we exchange $\rho'_1$ with $\rho'_2$ ($\rho''_1$ with $\rho''_2$) 
or $r'_1$ with $r'_2$ ($r''_1$ with $r''_2$)  
due to the symmetry of the right-hand sides of $(3.2)$ ($(3.3)$). 
In fact, the solutions $k^2=\rho^2_2/\rho^2_1$ lead to the same gauge factors as the solutions $k^2=1$. 
Therefore, we drop the solutions $k^2=\rho^2_2/\rho^2_1$ and keep $k^2=1$. 
This reasoning can also be applied to the remaining cases. 
In this way we conclude that there are in total 8 cases of parameterizations.
These parameterizations correspond to the following 8 classes of quasi-polynomial eigenfunctions.
\end{rmk}

\begin{theorem}
The Bannai-Ito operator has 8 sequences of quasi-polynomial eigenfunctions: 
$\{\xi_d(x)p^{(d)}_m(x)\}_{m=0}^{\infty}$, $d\in\{1,2,\ldots,8\}$.  
The 8 gauge factors are 
\begin{eqnarray*}
\xi_1(x) &=& 1, \\
\xi_2(x) &=& \frac{\Gamma\left(1/2+r_{1}+x \right)\Gamma \left(1/2+r_{1}-x\right)\Gamma\left(1/2+r_{2}+x \right)\Gamma \left(1/2+r_{2}-x\right)}
{\Gamma \left( \rho_{1}-x \right) \Gamma \left( 1+\rho_{1}+x \right) \Gamma \left( \rho_{2}-x \right) \Gamma \left( 1+\rho_{2}+x \right)},\\
\xi_3(x) &=& \frac{1}{\Gamma \left( \rho_{1}-x \right) \Gamma \left( 1+\rho_{1}+x \right) \Gamma \left( \rho_{2}-x \right)\Gamma\left( 1+\rho_{2}+x \right)}, \\
\xi_4(x) &=& \Gamma\left(1/2+r_{1}+x \right)\Gamma\left(1/2+r_{1}-x \right) \Gamma \left(1/2+r_{2}+x\right)\Gamma \left(1/2+r_{2}-x \right), \\
\xi_5(x) &=& \frac{\Gamma \left( 1/2+r_{1}+x \right) \Gamma\left( 1/2+r_{1}-x\right)}
{\big[\Gamma \left(\rho_{1}-x \right)\Gamma\left(1+\rho_{1}+x\right)}, \\
\xi_6(x) &=& \frac{\Gamma\left(1/2+r_{2}+x\right)\Gamma\left(1/2+r_{2}-x\right)}
{\Gamma\left(\rho_{2}-x\right)\Gamma\left(1+\rho_{2}+x\right)}, \\
\xi_7(x) &=& \frac{\Gamma \left( 1/2+r_{1}+x \right) \Gamma\left( 1/2+r_{1}-x\right)}
{\Gamma \left(\rho_{2}-x \right)\Gamma\left(1+\rho_{2}+x\right)}, \\
\xi_8(x) &=& \frac{\Gamma\left(1/2+r_{2}+x\right)\Gamma\left(1/2+r_{2}-x\right)}
{\Gamma\left(\rho_{1}-x\right)\Gamma\left(1+\rho_{1}+x\right)}.
\end{eqnarray*}
And the polynomials $p^{(d)}_m(x)=B_n(x;\sigma_{d})$ for $m\in\{0,1,2,\ldots\}$, $d\in\{1,2,\ldots,8\}$.  
Moreover, the eigenvalues of the quasi-polynomial eigenfunctions $\{\xi_d(x)p^{(d)}_m(x)\}$ are 
\vspace{-2mm}
\begin{empheq}[left={\mu_{\bf d}=\mu_{d,m}=\empheqlbrace}]{align}
\lambda_{m}(\sigma_d)+C_d,\hspace{2mm} & \quad \text{ if } d\in\{1,2,3,4\},  \\
-\lambda_{m}(\sigma_d)+C_d, & \quad \text{ if } d\in\{5,6,7,8\}, 
\end{empheq}
where $\lambda_{m}(\sigma_d)$ is the eigenvalue of the Bannai-Ito polynomial $B_m(x)$ (refers to (1.3), (1.4)) 
with parameters given by $\sigma_d$, 
and the definition of the constant $C_d$ is 
\begin{empheq}[left={C_d=\empheqlbrace}]{align}
-\frac{\rho_1+\rho_2-\rho'_1-\rho'_2}{2}+\frac{r_1+r_2-r'_1-r'_2}{2}, \hspace{7.8mm} & \hspace{2mm}
\text{ if } \{\rho'_1,\rho'_2,r'_1,r'_2\}=\sigma_d \text{ and } d\in\{1,2,3,4\},  \nonumber \\
-\frac{\rho_1+\rho_2+\rho''_1+\rho''_2}{2}+\frac{r_1+r_2+r''_1+r''_2-1}{2}, & \hspace{2mm} 
\text{ if } \{\rho''_1,\rho''_2,r''_1,r''_2\}=\sigma_d \text{ and } d\in\{5,6,7,8\}.  \nonumber
\end{empheq}
\end{theorem}

\begin{proof}
Here we derive the above quasi-polynomial eigenfunctions using Lemma 3.1.
Notice that $\xi(-x)/\xi(-x-1)=F(x)$ leads to $\xi(x+1)/\xi(x)=F(-x-1)$ if one substitute $x$ by $-x-1$, 
hence it is convenient to construct the gauge factors with the help of the Gamma function. 

\par
In the first case we have $\xi_1(x+1)/\xi_1(x)=1$, which implies $\xi_1(x)$ is a constant.

\par
In the second case we have
$$
\frac{\xi_2(x+1)}{\xi_2(x)}=
\frac{(x-\rho_1+1)(x-\rho_2+1)(x+r_1+1/2)(x+r_2+1/2)}{(x+\rho_1+1)(x+\rho_2+1)(x-r_1+1/2)(x-r_2+1/2)}.
$$
Recall that $\Gamma(x+1)/\Gamma(x)=x$, we can express $\xi_2(x)$ in terms of the Gamma functions
$$
\xi_2(x)=c_2(x)
\frac{\Gamma(x-\rho_1+1)\Gamma(x-\rho_2+1)\Gamma(x+r_1+1/2)\Gamma(x+r_2+1/2)}
{\Gamma(x+\rho_1+1)\Gamma(x+\rho_2+1)\Gamma(x-r_1+1/2)\Gamma(x-r_2+1/2)},
$$
where $c_2(x)$ is periodic function of period 1, $c_2(x+1)=c_2(x)$. 
The fact that $\xi_2(x)$ must be an eigenfunction of $H$ implies that $c_2(x)$ cannot be a constant. 
If we let
$$
c_2(x)=
\frac{\sin[\pi(x-\rho_1+1)]\sin[\pi(x-\rho_2+1)]}{\sin[\pi(x-r_1+1/2)]\sin[\pi(x-r_2+1/2)]},
$$
then the following expression of $\xi_2(x)$ satisfies $H[\xi_2(x)]=(r_1+r_2-\rho_1-\rho_2)\xi_2(x)$.
$$
\xi_2(x)=
\frac{\Gamma(x+r_1+1/2)\Gamma(-x+r_1+1/2)\Gamma(x+r_2+1/2)\Gamma(-x+r_2+1/2)}
{\Gamma(x+\rho_1+1)\Gamma(-x+\rho_1)\Gamma(x+\rho_2+1)\Gamma(-x+\rho_2)}
$$
In the derivation of the above expression of $\xi_2(x)$, Euler's reflection formula 
$$
\Gamma(x)\Gamma(1-x)=\pi/\sin(\pi x), \hspace{2mm}
x\notin\mathbb{Z}
$$
has been used (assuming $x-\rho_1+1$, $x-\rho_2+1$, $x-r_1+1/2$, $x-r_2+1/2\notin\mathbb{Z}$). 

\par
The remaining $\xi_k(x)$'s can be obtained in the same way 
by choosing suitable periodic function $c(x)$'s and applying Euler's reflection formula. 
We shall assume that the restriction $x\notin\mathbb{Z}$ is always satisfied wherever Euler's reflection formula 
was applied. This is not difficult since $\rho_1, \rho_2, r_1, r_2$ can be any real numbers.

\par
According to (3.1), (3.2), (3.3) and Lemma 3.1, 
it is easily seen that the conjugated operator 
\begin{empheq}[left={\tilde{H}=\empheqlbrace}]{align}
H(\sigma_d)+C_d,\hspace{2mm} & \quad \text{ if } d\in\{1,2,3,4\},  \nonumber \\
-H(\sigma_d)+C_d, & \quad  \text{ if } d\in\{5,6,7,8\},  \nonumber
\end{empheq}
thus $p^{(d)}_n(x)=B_n(x;\sigma_{d})$ and (3.8), (3.9) 
follow immediately. 
\end{proof}

\section{Exceptional Bannai-Ito polynomials}
\par
Using the results of Section 2 and Section 3, we are now able to construct the exceptional Bannai-Ito polynomials. 
We first show that there are missing degrees in the constructed exceptional Bannai-Ito polynomial sequences. 
Notably their missing degrees demonstrate different rules compared with the known 1-step XOPs.
And then we prove that the exceptional Bannai-Ito polynomials are orthogonal with respect 
to a discrete measure on the exceptional Bannai-Ito grid.

\par
Define an index set ${\bf D}=\{1,2,\ldots,8\}\times\mathbb{Z}_{\geq 0}$. 
For the sake of simplicity, we assume that for any index ${\bf d}=(d,m)\in{\bf D}$, 
a quasi-polynomial eigenfunction $\phi_{\bf d}(x)=\xi_{d}(x)p^{(d)}_m(x)$ is uniquely determined 
upon the constant multiplier. 
Last but not least, we assume that the Bannai-Ito polynomials mentioned in this paper are always monic 
(i.e. the coefficient of the highest order term is 1), hence 
$p^{(d)}_n(x)$ are always monic too. One could refer to Theorem 3.3.

\par
Now we take a quasi-polynomial eigenfunction $\phi_{\bf d}(x)$ with ${\bf d}=(d,m)$ as a seed solution and show that 
the Darboux transformed eigenfunction $\hat{\mathcal{F}}_{\phi_{\bf d}}[B_n(x)]$ is just the exceptional 
Bannai-Ito polynomial we want.
Firstly, a well selected decoupling coefficient $r(x)$ is essential.
\begin{lemma}
Assume that the decoupling coefficient $r(x)$ is given by
\begin{equation}
\label{eq:r}
r(x)=
\tilde{\chi}(x)\chi(x)\frac{\eta_{d}(x)}{x\xi_{d}(x)},
\end{equation}
where $\eta_d(x)$ is the polynomial of lowest degree such that 
$$
\frac{\xi_d(-x)}{\xi_d(x)}=\frac{\eta_{d}(-x)}{\eta_{d}(x)}.
$$
Then the Darboux transformed eigenfunction
$$
\hat{\mathcal{F}}_{\phi_{\bf d}}[B_n(x)]=\frac{r(x)}{\chi(x)\tilde{\chi}(x)}
\left|
    \begin{array}{ll}
      (I-R)[\phi_{\bf d}(x)] & (I-R)[B_n(x)] \\
      (I+TR)[\beta(-x-1)\phi_{\bf d}(x)] & (I+TR)[\beta(-x-1)B_n(x)]
    \end{array}
  \right|
$$
is a polynomial.
\end{lemma}

\begin{proof}
Given the seed solution 
$\phi_{\bf d}(x)=\xi_{d}(x)p^{(d)}_m(x)$ whose eigenvalue is $\mu_{\bf d}$, we have
$$\hspace{-30mm}
\frac{\chi(x)\tilde{\chi}(x)}{r(x)}\hat{\mathcal{F}}_{\phi_{\bf d}}[B_n(x)] =
(\phi_{\bf d}(x)-\phi_{\bf d}(-x))(\beta(-x-1)B_n(x)+\beta(x)B_n(-x-1))
$$
$$\hspace{8mm}
-(B_n(x)-B_n(-x))(\beta(-x-1)\phi_{\bf d}(x)+\beta(x)\phi_{\bf d}(-x-1))
$$
$$\hspace{2mm}
=(\phi_{\bf d}(x)-\phi_{\bf d}(-x))[(\lambda_n+\beta(x)+\beta(-x-1))B_n(x)+\alpha(x)(B_n(x)-B_n(-x))]
$$
$$\hspace{9mm}
-(B_n(x)-B_n(-x))[(\mu_{\bf d}+\beta(x)+\beta(-x-1))\phi_{\bf d}(x)+\alpha(x)(\phi_{\bf d}(x)-\phi_{\bf d}(-x))]
$$
$$\hspace{-34mm}
=(\lambda_n+\beta(x)+\beta(-x-1))B_n(x)(\phi_{\bf d}(x)-\phi_{\bf d}(-x))
$$
$$\hspace{-27mm}
-(\mu_{\bf d}+\beta(x)+\beta(-x-1))\phi_{\bf d}(x)(B_n(x)-B_n(-x))
$$
$$\hspace{-1mm}
=\frac{x\xi_{d}(x)}{\eta_{d}(x)}\bigg[
(\lambda_n+\beta(x)+\beta(-x-1))B_n(x)\big(\frac{\eta_{d}(x)p^{(d)}_m(x)-\eta_{d}(-x)p^{(d)}_m(-x)}{x}\big)
$$
$$\hspace{-14mm}
-(\mu_{\bf d}+\beta(x)+\beta(-x-1))\eta_{d}(x)p^{(d)}_m(x)\big(\frac{B_n(x)-B_n(-x)}{x}\big)\bigg],
$$
where the second equation follows from equation (2.3) which can be used as 
\begin{eqnarray*}
\beta(-x-1)\phi_{\bf d}(x)+\beta(x)\phi_{\bf d}(-x-1)
\hspace{-2.4mm}&=&\hspace{-2.4mm}
(\mu_{\bf d}+\beta(x)+\beta(-x-1))\phi_{\bf d}(x)+\alpha(x)(\phi_{\bf d}(x)-\phi_{\bf d}(-x)), \\
\beta(-x-1)B_n(x)+\beta(x)B_n(-x-1)
\hspace{-2.4mm}&=&\hspace{-2.4mm}
(\lambda_n+\beta(x)+\beta(-x-1))B_n(x)+\alpha(x)(B_n(x)-B_n(-x)).
\end{eqnarray*}
The third equation is obtained after the elimination of 
$$
\alpha(x)(B_n(x)-B_n(-x))(\phi_{\bf d}(x)-\phi_{\bf d}(-x)).
$$
Finally, by introducing the polynomial $\eta_{d}(x)$ we arrive at the fourth equation. 
Notice that the function in the square brackets is just the desired polynomial. 
Therefore, the decoupling coefficient should be
$$
r(x)=\tilde{\chi}(x)\chi(x)\frac{\eta_{d}(x)}{x\xi_{d}(x)}.
$$
\vspace{-2mm}
\end{proof}

\par
Here we give all the functions $\eta_d(x)$ explicitly for later convenience. 
This list is obtained via Lemma 3.1, 
\begin{center}
  \begin{tabular}{ll}
$\eta_1(x)=1$,  & $\eta_2(x)=(x-\rho_1)(x-\rho_2)$, \\
$\eta_3(x)=(x-\rho_1)(x-\rho_2)$,  & $\eta_4(x)=1$, \\
$\eta_5(x)=(x-\rho_1)$,  & $\eta_6(x)=(x-\rho_2)$, \\
$\eta_7(x)=(x-\rho_2)$,  & $\eta_8(x)=(x-\rho_1)$. \\
  \end{tabular}
\end{center}

\par
With $r(x)$ given by Lemma 4.1, let $B^{(1)}_{{\bf d},n}(x)=\hat{\mathcal{F}}_{\phi_{\bf d}}[B_{n}(x)]$, 
then $\{B^{(1)}_{{\bf d},n}(x)\}$ are the polynomial eigenfunctions of $\hat{H}^{(1)}$:
$\hat{H}^{(1)}[B^{(1)}_{{\bf d},n}(x)]=\lambda_nB^{(1)}_{{\bf d},n}(x)$ ($\lambda_n\neq\mu_{\bf d}$).
In what follows, we give an analysis of the degree of $B^{(1)}_{{\bf d},n}(x)$ 
and show that there are missing degrees in their polynomial sequences.
From the proof of Lemma 4.1, we have
\begin{equation}
\label{eq: B1}
B^{(1)}_{{\bf d},n}(x)=
(\lambda_n-\beta)\bigg(\frac{\eta_{d}(x)p^{(d)}_m(x)-\eta_{d}(-x)p^{(d)}_m(-x)}{x}\bigg)B_n(x)
\end{equation}
$$\hspace{8mm}
-(\mu_{\bf d}-\beta)\bigg(\frac{B_n(x)-B_n(-x)}{x}\bigg)\eta_{d}(x)p^{(d)}_m(x),
$$
where the constant $\beta=-(\beta(x)+\beta(-x-1))=r_1+r_2$ (refers to (2.22)). 

\begin{rmk}
Note that $\mu_{4,0}=r_1+r_2=\beta$ (refers to (3.8)), 
in this case the polynomials $B^{(1)}_{4,0,n}(x)=0$ for $n=0,1,2,\ldots$,
which implies that ${\bf d}=(4,0)$ cannot be chosen as the index of a seed solution.
\end{rmk}

\par
In order to see the degree of $B^{(1)}_{{\bf d},n}(x)$, let 
\begin{equation}
\label{Bn}
B_n(x)=x^n+a_{n-1}x^{n-1}+\cdots
\end{equation}
then from Theorem 3.3 we have
$$
p^{(d)}_n(x)=B_n(x;\sigma_{d})=x^n+a_{n-1}(\sigma_{d})x^{n-1}+\cdots
$$
where $a_{i}(\sigma_{d})$ ($i=0,1,\ldots,n-1$) are the coefficients with respect to the parameterization $\sigma_{d}$.
For $d=1,\ldots,8$, denote the degree of $\eta_{d}(x)$ by $\kappa_{d}$ such that
$
\eta_{d}(x)=x^{\kappa_{d}}+b^{(d)}_{\kappa_{d}-1}x^{\kappa_{d}-1}+\cdots.
$
Then 
\begin{equation}
\label{etaP}
\eta_{d}(x)p^{(d)}_m(x)=x^{m+\kappa_{d}}+(a_{m-1}(\sigma_{d})+b^{(d)}_{\kappa_{d}-1})x^{m+\kappa_{d}-1}+\cdots,
\end{equation}
where $\kappa_{d}=0$ for $d\in\{1,4\}$; $\kappa_{d}=2$ for $d\in\{2,3\}$; 
and $\kappa_{d}=1$ for $d\in\{5,6,7,8\}$.

\begin{prop}
Let $X$ be the set of degrees of $\{B^{(1)}_{{\bf d},n}(x)\}$, where the index ${\bf d}\in {\bf D}\backslash \{(1,n), (4,0)\}$. 
If the conditions (4.14) and (4.15)
are satisfied, 
then we have
\begin{itemize}
\item for $m$ is odd,
\begin{itemize}
\item[$\bullet$] $d=1:$  $X=\{m-1, m, m+1, \cdots, 2m-2, 2m, \cdots\};$
\item[$\bullet$] $d=4:$  $X=\{m-1, m, m+1, m+2, m+3, \cdots\};$
\item[$\bullet$] $d=2,3:$  $X=\{m+1, m+2, m+3, m+4, m+5, \cdots\};$
\item[$\bullet$] $d=5,6,7,8:$  $X=\{m-1, m+1, m+1, m+3, m+3, \cdots\};$
\end{itemize}
\item for $m$ is even, 
\begin{itemize}
\item[$\bullet$] $d=1:$ $X=\{m-2, m, m, \cdots, 2m-4, 2m-4, 2m-2, 2m, 2m, \cdots\};$
\item[$\bullet$] $d=4:$  $X=\{m-2, m, m, m+2, m+2, \cdots\};$
\item[$\bullet$] $d=2,3:$  $X=\{m, m+2, m+2, m+4, m+4, \cdots\};$
\item[$\bullet$] $d=5,6,7,8:$  $X=\{m, m+1, m+2, m+3, m+4, \cdots\}.$
\end{itemize}
\end{itemize}
\end{prop}

\begin{proof}
According to (\ref{Bn}) and (\ref{etaP}), the right-hand side of (\ref{eq: B1}) can be expanded as
$$
B^{(1)}_{{\bf d},n}(x)=
(\lambda_n-\beta)[x^{m+\kappa_{d}-1}+(-x)^{m+\kappa_{d}-1}+
(a_{m-1}(\sigma_{d})+b_{\kappa_{d}-1})(x^{m+\kappa_{d}-2}+(-x)^{m+\kappa_{d}-2})+\cdots]
(x^n+a_{n-1}x^{n-1}+\cdots)
$$
$$\hspace{16mm}
-(\mu_{\bf d}-\beta)
[x^{n-1}+(-x)^{n-1}+a_{n-1}(x^{n-2}+(-x)^{n-2})+\cdots]
[x^{m+\kappa_{d}}+(a_{m-1}(\sigma_{d})+b_{\kappa_{d}-1})x^{m+\kappa_{d}-1}+\cdots]
$$
$$\hspace{-3mm}
=(\lambda_n-\beta)[x^{m+\kappa_{d}-1}+(-x)^{m+\kappa_{d}-1}
+(a_{m-1}(\sigma_{d})+b_{\kappa_{d}-1})(x^{m+\kappa_{d}-2}+(-x)^{m+\kappa_{d}-2})]x^{n}
$$
$$\hspace{-18mm}
-(\mu_{\bf d}-\beta)
[x^{n-1}+(-x)^{n-1}+a_{n-1}(x^{n-2}+(-x)^{n-2})]x^{m+\kappa_{d}}
+\cdots.
$$
Note that $x^k+(-x)^k=2x^k$ if $k$ is even, otherwise $x^k+(-x)^k=0$.
The degree of $B^{(1)}_{{\bf d},n}(x)$ depends on the parities of $m+\kappa_{d}$ and $n$.
Specifically, the leading term of $B^{(1)}_{{\bf d},n}(x)$ is given by:
\begin{empheq}[left={B^{(1)}_{{\bf d},n}(x)=\empheqlbrace}]{align}
2(\lambda_n-\beta)x^{n+m+\kappa_{d}-1}+\cdots, \hspace{1.8mm} & \quad 
 \text{ if } m+\kappa_{d} \text{ is odd and } n \text{ is even,} \nonumber \\
2(\lambda_n-\mu_{\bf d})x^{n+m+\kappa_{d}-1}+\cdots,  & \quad 
\text{ if } m+\kappa_{d} \text{ is odd and } n \text{ is odd,} \nonumber \\
2C_{d,m,n}x^{n+m+\kappa_{d}-2}+\cdots,  \hspace{4.8mm} & \quad 
\text{ if } m+\kappa_{d} \text{ is even and } n \text{ is even,} \nonumber  \\
-2(\mu_{\bf d}-\beta)x^{n+m+\kappa_{d}-1}+\cdots,  & \quad 
\text{ if } m+\kappa_{d} \text{ is even and } n \text{ is odd,} \nonumber 
\end{empheq}
where $C_{n,m,d}=(\lambda_n-\beta)(a_{m-1}(\sigma_{d})+b^{(d)}_{\kappa_{d}-1})-(\mu_{\bf d}-\beta)a_{n-1}$.
Recall that in the case $d=1$, the seed solution is $\phi_{(1,m)}=B_m(x)$, thus 
$B^{(1)}_{{\bf d},n}(x)=0$ for $d=1$, $n=m$ since $\hat{\mathcal{F}}_{\phi_{(1,m)}}[B_m(x)]=0$. 
In the other cases if we assume that the coefficient of the leading term of $B^{(1)}_{{\bf d},n}(x)$ is always non-zero, i.e.
\begin{eqnarray}
& & \lambda_n-\beta \neq 0, \hspace{2mm} \text{ if } n\text{ is even, }  \\
& & \mu_{\bf d}-\beta \neq 0, \hspace{2mm} \text{ if } d\in\{1,2,3,4\} \text{ and } m\text{ is even; or } 
d\in\{5,6,7,8\} \text{ and } m\text{ is odd, } \\
& & \lambda_n-\mu_{\bf d}\neq 0, \hspace{2mm} \text{ if } d\in\{1,2,3,4\}, m\text{ is odd and } n \text{ is odd; } \\
& & \hspace{2.1cm} \text{or }d\in\{5,6,7,8\}, m\text{ is even and } n \text{ is odd, } \nonumber \\
& & C_{d,m,n}\neq 0, \hspace{2mm} \text{ if }d\in\{1,2,3,4\}, m\text{ is even and } n \text{ is even; }  \\
& & \hspace{1.8cm}\text{or } d\in\{5,6,7,8\}, m\text{ is odd and } n \text{ is even,} \nonumber 
\end{eqnarray}
are satisfied, then for the degree of $B^{(1)}_{{\bf d},n}(x)$ we have
\vspace{-2mm}
\begin{empheq}[left={\deg{(B^{(1)}_{{\bf d},n}(x))}=\empheqlbrace}]{align}
n+m+\kappa_{d}-1, & \quad \text{ if } m+\kappa_{d} \text{ is odd; } 
\text{ or } m+\kappa_{d} \text{ is even and } n \text{ is odd,} \\
n+m+\kappa_{d}-2, & \quad \text{ if } m+\kappa_{d} \text{ is even and } n \text{ is even.} 
\end{empheq}
By analyzing the parity of $m$ and $n$, it turns out that (4.9) leads to the first three cases for $m$ is odd 
and the last case for $m$ is even, 
while (4.10) leads to the other cases of Proposition 4.3.

\par
Finally, let us consider under which conditions will (4.5)-(4.8) be satisfied. 
It follows from (1.3), (1.4) and (2.23) that 
\vspace{-2mm}
\begin{empheq}[left={\lambda_n-\beta=\empheqlbrace}]{align}
n/2-(r_1+r_2), \hspace{10.8mm} & \quad \text{ if } n \text{ is even, }  \\
-(\rho_1+\rho_2)-(n+1)/2, & \quad  \text{ if } n \text{ is odd,} 
\end{empheq}
thus (4.5) holds under the condition 
\begin{equation}
\label{cond 1}
r_1+r_2 \notin \mathbb{Z}.
\end{equation} 
In the Appendix A we list the explicit expressions for $\mu_{\bf d}-\beta$, $\lambda_n-\mu_{\bf d}$ and $C_{d,m,n}$ 
in the corresponding cases. From (A.1)-(A.24) we can see that (4.6)-(4.8) are satisfied under the conditions 
\begin{equation}
\label{cond 2}
r_1+r_2, \hspace{2mm} \rho_1+\rho_2, \hspace{2mm} r_1+r_2+\rho_1+\rho_2, \hspace{2mm} 
r_1+r_2-\rho_1-\rho_2
\notin \mathbb{Z}, 
\end{equation} 
\begin{equation}
\label{cond 3}
\frac{r_1-r_2-\rho_1+\rho_2+1}{2}, \hspace{2mm} \frac{r_1-r_2+\rho_1-\rho_2+1}{2}, \hspace{2mm} 
r_i-\rho_j+\frac{1}{2}, \hspace{2mm} r_i+\rho_j+\frac{1}{2}
\notin \mathbb{Z}, \hspace{2mm} i,j\in\{1,2\},
\end{equation} 
and the index ${\bf d}=(d,m)\notin \{(1,n), (4,0)\}$.

\end{proof}

\begin{rmk}
As we addressed in the proof of Proposition 4.3, in both cases $d=1$ there are missing degrees ($2m-1$ and $2m-2$, respectively) in the degree sequence of $\{B^{(1)}_{{\bf d},n}(x)\}$ when $n=m$, 
since the seed solutions $\phi_{(1,m)}=B_m(x)$ and the trivial eigenfunctions 
$\hat{\mathcal{F}}_{\phi_{(1,m)}}[B_m(x)]=0$. 
On the other hand, the nontrivial eigenfunction given by $(\ref{DTeigenfunction})$ becomes
$$
\phi_{\bf d}^{(1)}(x)=
\frac{\sigma(x)r(x)}{\tilde{\chi}(x)\chi(x)\alpha(x)\omega(x)}=
\frac{\sigma(x)\eta_{d}(x)}{x\alpha(x)\omega(x)\xi_{d}(x)}.
$$
Here we may choose 
$\sigma(x)=x\alpha(x)\omega(x)\xi_{1}(x)/\eta_{1}(x)$. 
With this choice it holds that $\sigma(x+1)=\sigma(x)$ and $\sigma(-x)=-\sigma(x)$, 
then 
$\phi_{\bf d}^{(1)}(x)=
\eta_{d}(x)\xi_{1}(x)/\eta_{1}(x)\xi_{d}(x),$
such that only when $d=1$ the nontrivial eigenfunction $\phi_{\bf d}^{(1)}(x)$ is a polynomial and 
$\phi_{(1,m)}^{(1)}(x)=1$.
This eigenfunction should be added to the sequence $\{B^{(1)}_{(1,m),n}(x)\}$, 
hence the polynomial sequence starts from degree $0$.
However, for the convenience of later discussion, we will not include this term into the XBI polynomial sequence.
\end{rmk}

\par
Notice that in each case there are missing degrees in the degree sequence $X$, 
and the degrees of $\{B^{(1)}_{{\bf d},n}(x)\}$ demonstrate exactly opposite features with regard to the parity of $m$. 
Specifically, in the first three cases when $m$ is odd and the in the last case when $m$ is even 
$X$ behaves similarly as the known 1-step XOPs, where $X$ is cofinite (the complement of $X$ is finite).
However, in the other cases $X$ is not cofinite and only contains even degrees. 
This fact implies that in these cases, the normalized 1-step exceptional Bannai-Ito operator $\hat{H}^{(1)}$ only has 
even-order eigenpolynomials,  and there are two different series of these eigenpolynomials. 
It would be better to say that the eigenpolynomials of odd degrees are not been deleted 
but are just replaced by the ones of even degrees.
For example, in the case $m=2$ and $d\in\{2,3\}$, the degree set of $\{B^{(1)}_{{\bf d},n}(x)\}$ is $X=\{2,4,4,6,6,\ldots\}$. 
The polynomials $B^{(1)}_{{\bf d},n}(x)$ are different from each other even when they have the same degree, 
and by no exception they are orthogonal according to Theorem 4.11. 

\par
The feature that $X$ is not cofinite naturally conflicts with the definition of XOPs 
which satisfy a second-order differential (difference) equation. 
This is not surprising, the constraint that $X$ should be cofinite is just the consequence of the way 
XOPs in \cite{ECH, EML, EHJ, GKM09, GKM10, 2DT, EH, BEOP, XLJ} are presented, 
the reflection operator $R$ has not appeared in their eigenvalue equations. 
Taking this opportunity, we would like to modify the definition of XOPs as the generalization of COPs 
where the only condition to be removed is that it contains polynomials of all degrees. 
By this means, the XOPs are characterized by orthogonality and  
forming the polynomial eigenfunctions of certain differential (difference) operators. 
From now on, we call $\{B^{(1)}_{{\bf d},n}(x)\}$ the 1-step exceptional Bannai-Ito (XBI) polynomials.

\begin{prop}
The exceptional Bannai-Ito polynomials $B^{(1)}_{{\bf d},n}(x)$ can be expressed as the linear combination of the 
Bannai-Ito polynomials $B_n(x)$ and $B_n(-x)-B_n(x)$,
$$
B^{(1)}_{{\bf d},n}(x)=C^{(1)}_{{\bf d},n}(x)B_n(x)+C^{(2)}_{{\bf d}}(x)(B_n(x)-B_n(-x)),
$$
where 
$$
C^{(1)}_{{\bf d},n}(x)=(\lambda_n-\beta)\bigg(\frac{\hat{\phi}_{\bf d}^{(0)}(x)-\hat{\phi}_{\bf d}^{(0)}(-x)}{x}\bigg), 
\hspace{2mm}
C^{(2)}_{{\bf d}}(x)=-(\mu_{\bf d}-\beta)\frac{\hat{\phi}_{\bf d}^{(0)}(x)}{x},
$$
and $\hat{\phi}_{\bf d}^{(0)}(x)$ is the normalized seed solution which is a polynomial:
\begin{equation}
\label{eq:normalized seed}
\hat{\phi}_{\bf d}^{(0)}(x)=\eta_{d}(x)p^{(d)}_m(x)=\phi_{\bf d}(x)\frac{\eta_{d}(x)}{\xi_{d}(x)}.
\end{equation} 
\end{prop}

\begin{proof}
The above result follows directly from (\ref{eq: B1}).
\end{proof}

\par
\begin{corollary}
The exceptional Bannai-Ito polynomials $B^{(1)}_{{\bf d},n}(x)$ satisfy
\begin{equation}
B^{(1)}_{{\bf d},n}(x)-B^{(1)}_{{\bf d},n}(-x)=
\frac{1}{x}(\lambda_n-\mu_{\bf d})(\hat{\phi}_{\bf d}^{(0)}(x)-\hat{\phi}_{\bf d}^{(0)}(-x))(B_n(x)-B_n(-x)).
\end{equation}
\end{corollary}

\begin{proof}
From Proposition 4.5, we have
\begin{eqnarray*}
B^{(1)}_{{\bf d},n}(x)-B^{(1)}_{{\bf d},n}(-x) \hspace{-2.4mm}&=&\hspace{-2.4mm}
C^{(1)}_{{\bf d},n}(x)B_n(x)-C^{(1)}_{{\bf d},n}(-x)B_n(-x) \\
\hspace{-2.4mm}& &\hspace{-2.4mm}
+(C^{(2)}_{{\bf d}}(x)+C^{(2)}_{{\bf d}}(-x))(B_n(-x)-B_n(x)).
\end{eqnarray*}
Since $C^{(1)}_{{\bf d},n}(x)=C^{(1)}_{{\bf d},n}(-x)$, it turns out that
\begin{eqnarray*}
B^{(1)}_{{\bf d},n}(x)-B^{(1)}_{{\bf d},n}(-x) \hspace{-2.4mm}&=&\hspace{-2.4mm}
(C^{(1)}_{{\bf d},n}(x)+C^{(2)}_{{\bf d}}(x)+C^{(2)}_{{\bf d}}(-x))(B_n(-x)-B_n(x)) \\
\hspace{-2.4mm}&=&\hspace{-2.4mm}
\frac{1}{x}(\lambda_n-\mu_{\bf d})(\hat{\phi}_{\bf d}^{(0)}(x)-\hat{\phi}_{\bf d}^{(0)}(-x))(B_n(-x)-B_n(x)).
\end{eqnarray*}
\end{proof}

\subsection{Orthogonality}
\par
A finite difference operator $L$ is said to be symmetric with respect to an inner product $\langle , \rangle_{\omega(x)}$
if it satisfies $\langle L[p(x)], q(x)\rangle_{\omega(x)} = \langle p(x), L[q(x)]\rangle_{\omega(x)} $ 
for any functions $p(x)$ and $q(x)$, where the inner product is defined by 
$\langle p(x),q(x)\rangle_{\omega(x)}=\sum_{x\in\chi}\omega(x)p(x)q(x)$.
It is known from the Lemma 2.4 of \cite{ECH} that if a difference operator is symmetric with respect to 
an inner product $\langle , \rangle_{\omega(x)}$, 
then its eigenfunctions are orthogonal with respect to $\omega(x)$.

\par
Assume that the operator $\mathcal{L}$ has polynomial eigenfunctions $\{p_n(x)\}$ 
$$
\mathcal{L}[p_n(x)]=\lambda_np_n(x) \hspace{2mm} 
(\lambda_n\neq\lambda_m, n\neq m),
$$
the linearity of the inner product 
$\langle , \rangle_{\omega(x)}$ implies
\begin{eqnarray*}
\langle \mathcal{L}[p_n(x)],p_m(x)\rangle_{\omega(x)} \hspace{-2.4mm}&=&\hspace{-2.4mm} 
\lambda_n\langle p_n(x),p_m(x)\rangle_{\omega(x)}, \\
\langle p_n(x),\mathcal{L}[p_m(x)]\rangle_{\omega(x)} \hspace{-2.4mm}&=&\hspace{-2.4mm} 
\lambda_m\langle p_n(x),p_m(x)\rangle_{\omega(x)}.
\end{eqnarray*}
Then if $\mathcal{L}$ is symmetric with respect to $\langle , \rangle_{\omega(x)}$, it holds that
$$
\langle p_n(x),p_m(x)\rangle_{\omega(x)}=0  \hspace{2mm} (n\neq m),
$$
which demonstrate the orthogonality of $\{p_n(x)\}$ with respect to $\omega(x)$.
In light of this conclusion we need first to derive the conditions for $\mathcal{L}$ to be symmetric.

\begin{lemma}
Let $\omega(x)$ be a weight function supported on a countable set $\chi\subset \mathbb{R}$. 
$\mathcal{L}$ is a Dunkl shift operator of the form $\mathcal{L}=F(x)R+G(x)TR+C(x)$. 
Assume that the functions $F(x)$, $G(x)$ and $\omega(x)$ satisfy the following relations
\begin{eqnarray}
\hspace{-2.5mm}& &\hspace{-2.5mm} \omega(x)F(x) =\omega(-x)F(-x), \label{eq:WC1} \\
\hspace{-2.5mm}& &\hspace{-2.5mm} \omega(x)G(x)=\omega(-x-1)G(-x-1),  \label{eq:WC2} 
\end{eqnarray}
for $x\in\mathbb{R}$, and the boundary conditions
\begin{eqnarray}
\hspace{-2.5mm}& &\hspace{-2.5mm}
F(x)=0 \hspace{2mm}\text{for } x\in \chi\backslash(-\chi),  \label{eq:BC1} \\
\hspace{-2.5mm}& &\hspace{-2.5mm}
G(x)=0 \hspace{2mm}\text{for } x\in \chi\backslash(-\chi-1),  \label{eq:BC2} 
\end{eqnarray}
then $\mathcal{L}$ is symmetric with respect to $\omega(x)$.
Here we denote by $-\chi$ and $-\chi-1$ the sets $-\chi=\{-x:x\in\chi\}$ and $-\chi-1=\{-x-1:x\in\chi\}$, respectively.
\end{lemma}

\begin{proof}
According to the conditions, it holds for any functions $p(x)$ and $q(x)$ that
$$
\sum_{x\in\chi}\omega(x)F(x)p(-x)q(x)=\sum_{x\in-\chi}\omega(-x)F(-x)p(x)q(-x)
$$
$$\hspace{1.2cm}
=\sum_{x\in (-\chi)\cap\chi}\omega(x)F(x)p(x)q(-x)
=\sum_{x\in\chi}\omega(x)F(x)p(x)q(-x),
$$
and
$$
\sum_{x\in\chi}\omega(x)G(x)p(-x-1)q(x)=\sum_{x\in(-\chi-1)}\omega(-x-1)G(-x-1)p(x)q(-x-1)
$$
$$\hspace{0cm}
=\sum_{x\in(-\chi-1)\cap\chi}\omega(x)G(x)p(x)q(-x-1)
=\sum_{x\in\chi}\omega(x)G(x)p(x)q(-x-1).
$$
These equations imply that
$$
\sum_{x\in\chi}\omega(x)\big[F(x)p(-x)+G(x)p(-x-1)\big]q(x)=
\sum_{x\in\chi}\omega(x)p(x)\big[F(x)q(-x)+G(x)q(-x-1)\big].
$$
The above equation is equivalent to
$
\langle \mathcal{L}[p(x)], q(x)\rangle_{\omega(x)} = \langle p(x), \mathcal{L}[q(x)]\rangle_{\omega(x)},
$
thus $\mathcal{L}$ is symmetric with respect to $\omega(x)$.
\end{proof}

\par
\begin{lemma}
Assume that $\omega(x)$ is the weight function associated with a Dunkl-shift operator 
$\mathcal{L}=F(x)R+G(x)TR+C(x)$ such that (\ref{eq:WC1}), (\ref{eq:WC2}), (\ref{eq:BC1}) and (\ref{eq:BC2}) hold, 
then it satisfies
$$
\frac{\omega(x+1)}{\omega(x)}=\frac{F(-x-1)G(x)}{F(x+1)G(-x-1)}.
$$
\end{lemma}

\begin{proof}
The equations
$$
\omega(x)F(x)=\omega(-x)F(-x),  \quad \omega(x)G(x)=\omega(-x-1)G(-x-1)
$$
imply that
$$
\frac{\omega(-x)}{\omega(-x-1)}=\frac{F(x)G(-x-1)}{F(-x)G(x)},
$$
after substituting $x$ by $-x-1$ then we get the desired result.
\end{proof}

\par
Now we are able to give the weight functions of the exceptional Bannai-Ito operator 
by using the properties of the Gamma function as we did in Section 3. 
However, it takes less steps if the relationship between the weight functions of 
the exceptional Bannai-Ito operator and the Bannai-Ito operator is available, 
since the weight functions of the Bannai-Ito polynomials are already known in \cite{DopBI}.
Below we give the weight function of the exceptional Bannai-Ito polynomials $\{B^{(1)}_{{\bf d},n}(x)\}$ 
and show their orthogonality explicitly. 
These results can be extended to the multiple-step exceptional Bannai-Ito polynomials.

\begin{theorem}
Let $\hat{\omega}^{(1)}(x)$ be the weight function associated with the exceptional Bannai-Ito 
operator $\hat{H}^{(1)}$, and $\omega(x)$ be the weight function of the Bannai-Ito operator $H$, then it holds that
\begin{equation}
\label{eq:weight of XBI}
\hat{\omega}^{(1)}(x)=c(x)\frac{\tilde{\chi}(x)\chi(x)\alpha(x)}{r^{2}(x)}\omega(x),
\end{equation}
where $c(x)$ is a 1-periodic function $c(x+1)=c(x)$, 
and it satisfies $c(-x)=c(x)$.
\end{theorem}

\begin{proof}
According to (2.25), (2.26), 
the normalized 1-step exceptional Bannai-Ito operator is
$$\hspace{-20mm}
\hat{H}^{(1)}=\frac{r(x)\tilde{\chi}(-x)}{r(-x)\chi(x)}(R-I)+\frac{r(x)\alpha(-x-1)\beta(x)\chi(-x-1)}{r(-x-1)\tilde{\chi}(x)}(TR-I)
$$
$$\hspace{-28mm}
+\frac{\alpha(-x-1)\beta(x)\chi(-x-1)}{\tilde{\chi}(x)}\bigg(\frac{r(x)}{r(-x-1)}-1\bigg).
$$
Then from Lemma 4.8, we have
$$
\frac{\hat{\omega}^{(1)}(x+1)}{\hat{\omega}^{(1)}(x)} =
\frac{\tilde{\chi}(x+1)\chi(x+1)r^2(x)\alpha(-x-1)\beta(x)}
{\tilde{\chi}(x)\chi(x)r^2(x+1)\alpha(x)\beta(-x-1)} 
=\frac{\tilde{\chi}(x+1)\chi(x+1)r^2(x)\alpha(x+1)\omega(x+1)}
{\tilde{\chi}(x)\chi(x)r^2(x+1)\alpha(x)\omega(x)},
$$
hence
$$
\hat{\omega}^{(1)}(x)=
c(x)\frac{\tilde{\chi}(x)\chi(x)\alpha(x)\omega(x)}{r^2(x)},
$$
where $c(x)$ is a periodic function of period 1, $c(x+1)=c(x)$. 
Moreover, if we check the relations (\ref{eq:WC1}) and (\ref{eq:WC2}) with respect to the coefficients of $\hat{H}^{(1)}$, 
it turns out that
$$
c(-x)=c(x), \hspace{2mm} c(-x-1)=c(x).
$$
Notice that $c(-x-1)=c(x)$ follows from the conditions $c(x+1)=c(x)$ and $c(-x)=c(x)$, 
thus only the latter two conditions are essential. 
In view of these properties, it may sometimes be convenient to choose $c(x)$ as a constant.
\end{proof}

\par
It remains to derive the exceptional Bannai-Ito grid 
(from the simple roots of exceptional Bannai-Ito polynomials), 
which varies in the choice of the seed solution.
It is known that the Bannai-Ito polynomials have simple and distinct real roots \cite{DopBI} 
as the other COPs do \cite{OP}. 
Specifically, when $n$ is odd, if we assume that $r_{2}=\rho_{2}+n/2$, then the Bannai-Ito polynomial $B_{n}(x)$ has 
$n$ simple roots given by
$$
\rho_2, \hspace{2mm} -\rho_2-1, \hspace{2mm} \rho_2+1, \hspace{2mm}\cdots,\hspace{2mm} 
-\rho_2-\frac{n-1}{2}, \hspace{2mm} \rho_2+\frac{n-1}{2};
$$
when $n$ is even, 
assume that $\rho_1=-\rho_{2}-n/2$, then the $n$ simple roots of $B_{n}(x)$ are
$$
\rho_2, \hspace{2mm} -\rho_2-1, \hspace{2mm} \rho_2+1, \hspace{2mm} -\rho_2-2, \hspace{2mm} \cdots,
\hspace{2mm} \rho_2+\frac{n-2}{2}, \hspace{2mm} -\rho_2-\frac{n}{2}.
$$
These roots can be rewritten into a more elegant form which was called the Bannai-Ito grid: 
$x_s=-1/4+(-1)^s(x_0+s/2+1/4)$ ($s=0,1,\ldots,n-1$) with $x_0=\rho_2$.

\begin{remark}
Note that in the case when $n$ is odd there are 4 possible conditions: $r_{i}-\rho_{j}=n/2$, $i,j=1,2$, 
where we just restrict with the condition $r_{2}=\rho_{2}+n/2$ for the sake of simplicity. 
And in the case when $n$ is even there are also 2 possible conditions: $\rho_1+\rho_2=-n/2$ and 
$r_1+r_{2}=n/2$, 
for the same reason we restrict with the condition $\rho_1+\rho_2=-n/2$.
More details about Bannai-Ito grid can be found in \cite{DopBI}.
\end{remark}

\par
It then follows from Proposition 4.5 and the above results that 
$B^{(1)}_{{\bf d},n}(x)=Q_{{\bf d},n}(x)B_n(x)$, where
\vspace{-1mm}
\begin{empheq}[left={Q_{{\bf d},n}(x)=\empheqlbrace}]{align}
C^{(1)}_{{\bf d},n}(x)+2xC^{(2)}_{\bf d}(x)/(x-\rho_2),\hspace{19.8mm} & \hspace{2mm}
\text{ if } n \text{ odd and }r_{2}=\rho_{2}+\frac{n}{2}, \nonumber \\
C^{(1)}_{{\bf d},n}(x)+nxC^{(2)}_{\bf d}(x)/((x-\rho_2)(x+\rho_2+\frac{n}{2})), & \hspace{2mm}
\text{ if } n \text{ even and }\rho_{1}=-\rho_{2}-\frac{n}{2}. \nonumber
\end{empheq}
Note that $C^{(1)}_{{\bf d},n}(x)$ is a polynomial in $x$, 
$2xC^{(2)}_{\bf d}(x)/(x-\rho_2)$ is a polynomial for $d=1,2,5,6$, 
and $nxC^{(2)}_{\bf d}(x)/(x-\rho_2)(x+\rho_2+\frac{n}{2})$ is a polynomial for $d=1,2$. 
In these cases, the roots of $B_n(x)$ belong to the simple roots of $B^{(1)}_{{\bf d},n}(x)$. 
In other cases only a part of the former belong to the latter.

\par
Let $\chi_{\bf d}^{(1)}$ be the set whose elements come from the exceptional Bannai-Ito grid, 
thus these elements are the simple roots of $B^{(1)}_{{\bf d},N}(x)$.
For simplicity's sake, we rewrite the normalized exceptional Bannai-Ito operator into
$$
\hat{H}^{(1)}=\hat{\alpha}^{(1)}(x)(R-I)+\hat{\beta}^{(1)}(x)(TR-I)+\hat{\gamma}^{(1)}(x),
$$
and consider the eigenvalue equation $\hat{H}^{(1)}[B^{(1)}_{{\bf d},N}(x)]=\lambda_NB^{(1)}_{{\bf d},N}(x)$.
Assume that $x^{(1)}_s\in \chi_{\bf d}^{(1)}$, then this eigenvalue equation becomes
$$
\hat{\alpha}^{(1)}(x^{(1)}_s)B^{(1)}_{{\bf d},N}(-x^{(1)}_s)+\hat{\beta}^{(1)}(x^{(1)}_s)B^{(1)}_{{\bf d},N}(-x^{(1)}_s-1)=0.
$$
If $x^{(1)}_s\in \chi_{\bf d}^{(1)}\backslash(-\chi_{\bf d}^{(1)})$ 
and $x^{(1)}_s\in \chi_{\bf d}^{(1)}\cap(-\chi_{\bf d}^{(1)}-1)$, 
which means $B^{(1)}_{{\bf d},N}(-x^{(1)}_s)\neq 0$ and $B^{(1)}_{{\bf d},N}(-x^{(1)}_s-1)=0$, 
then $\hat{\alpha}^{(1)}(x^{(1)}_s)=0$.
On the other hand, if $x^{(1)}_s\in \chi_{\bf d}^{(1)}\backslash(-\chi_{\bf d}^{(1)}-1)$ 
and $x^{(1)}_s\in \chi_{\bf d}^{(1)}\cap(-\chi_{\bf d}^{(1)})$, 
which means $B^{(1)}_{{\bf d},N}(-x^{(1)}_s-1)\neq 0$ and $B^{(1)}_{{\bf d},N}(-x^{(1)}_s)=0$, 
then $\hat{\beta}^{(1)}(x^{(1)}_s)=0$.
Using these results together with the boundary conditions in Lemma 4.7, we can conclude that
$$
\chi_{\bf d}^{(1)}\backslash(-\chi_{\bf d}^{(1)}) \subseteq \chi_{\bf d}^{(1)}\cap(-\chi_{\bf d}^{(1)}-1), \hspace{2mm}
\chi_{\bf d}^{(1)}\backslash(-\chi_{\bf d}^{(1)}-1) \subseteq \chi_{\bf d}^{(1)}\cap(-\chi_{\bf d}^{(1)}),
$$
hence 
$\chi_{\bf d}^{(1)}=\chi_{\bf d}^{(1)}\cap(-\chi_{\bf d}^{(1)})\cap(-\chi_{\bf d}^{(1)}-1)
+\chi_{\bf d}^{(1)}\backslash(-\chi_{\bf d}^{(1)})+\chi_{\bf d}^{(1)}\backslash(-\chi_{\bf d}^{(1)}-1)$.
The first set in the right-hand side consists of part of the roots of $B_N(x)$, 
while the remaining $2$ sets can be obtained from the simple roots of $\hat{\alpha}^{(1)}(x)$ and $\hat{\beta}^{(1)}(x)$.

\begin{theorem}
The exceptional Bannai-Ito polynomials $\{B^{(1)}_{{\bf d},n}(x)\}$ satisfy the discrete orthogonality relation
\begin{equation}
\label{eq:orthogonality of XBI}
\sum_{x\in\chi_{\bf d}^{(1)}}\hat{\omega}^{(1)}(x)B^{(1)}_{{\bf d},n}(x)B^{(1)}_{{\bf d},m}(x)
=h^{(1)}_{{\bf d},n}\delta_{nm} \hspace{2mm} 
(0\leq n,m<N),
\end{equation}
where $h^{(1)}_{{\bf d},n}$ is constant.
If $N$ is odd and $r_2=\rho_2+N/2$, 
$$
d=1,5: \chi_{\bf d}^{(1)}=\{x_1,\ldots,x_{N-1}\}, \quad
d=2,6: \chi_{\bf d}^{(1)}=\{x_0,\ldots,x_{N}\},
$$
$$
d=3,7: \chi_{\bf d}^{(1)}=\{x_0,\ldots,x_{N-1}\}, \quad
d=4,8: \chi_{\bf d}^{(1)}=\{x_1,\ldots,x_{N}\};
$$
if $N$ is even and $\rho_1=-\rho_2-N/2$,
$$
d=1,4: \chi_{\bf d}^{(1)}=\{x_1,\ldots,x_{N-2}\}, \quad
d=2,3: \chi_{\bf d}^{(1)}=\{x_0,\ldots,x_{N-1}\},
$$
$$
d=5,8: \chi_{\bf d}^{(1)}=\{x_1,\ldots,x_{N-1}\}, \quad
d=6,7: \chi_{\bf d}^{(1)}=\{x_0,\ldots,x_{N-2}\},
$$
where $x_s=-1/4+(-1)^s(\rho_2+s/2+1/4)$.
\end{theorem}

\begin{rmk}
It should be noted that the weight $\hat{\omega}^{(1)}(x_s)$ is not always positive. 
From the formula (\ref{eq:weight of XBI}) we know that the positivity of $\hat{\omega}^{(1)}(x_s)$ depends on 
that of $\tilde{\chi}(x_s)\chi(x_s)\alpha(x_s)/r^2(x_s)$. 
In fact, we can choose a positive 1-periodic function $c(x)$, 
besides, the Bannai-Ito weight $\omega(x_s)$ is positive by construction \cite{DopBI}. 
Let us rewrite this expression into
$$
\frac{\tilde{\chi}(x)\chi(x)\alpha(x)}{r^2(x)}=
\frac{x^2\xi^2_{d}(x)\alpha(x)}{\eta^2_{d}(x)\tilde{\chi}(x)\chi(x)}
$$
$$
=\frac{x}{\eta_{d}(x)p^{(d)}_{m}(x)-\eta_{d}(-x)p^{(d)}_{m}(-x)}
\cdot
\frac{x\alpha(x)}{\eta_{d}(x)\big(\beta(-x-1)p^{(d)}_{m}(x)+\beta(x)\frac{\xi_{d}(-x-1)}{\xi_{d}(x)}p^{(d)}_{m}(-x-1)\big)}.
$$
It is convenient to discuss the positivity of $E_{\bf d}(x):=E^{(1)}_{d}(x)E^{(2)}_{\bf d}(x)E^{(3)}_{\bf d}(x)$ instead:
$$
E^{(1)}_{d}(x):=\frac{x\alpha(x)}{\eta_{d}(x)}, \hspace{2mm}
E^{(2)}_{\bf d}(x):=\frac{\eta_{d}(x)p^{(d)}_{m}(x)-\eta_{d}(-x)p^{(d)}_{m}(-x)}{x},
$$
$$
E^{(3)}_{\bf d}(x):=
\beta(-x-1)p^{(d)}_{m}(x)+\beta(x)\frac{\xi_{d}(-x-1)}{\xi_{d}(x)}p^{(d)}_{m}(-x-1).
$$
That is, the positivity of the weight $\hat{\omega}^{(1)}(x_s)$ is equivalent with that of $E_{\bf d}(x_s)$. 
Consider the case $m=1$, $N$ is odd and $r_2=\rho_2+N/2$, the expressions of $E_{\bf d}(x)$ follow as:
\begin{eqnarray}
E_{1,1}(x) \hspace{-2.4mm}&=&\hspace{-2.4mm}
(x-\rho_1)(x-\rho_2)\bigg[(x+\frac{1}{2})^2+
\frac{(\rho_{1}+\frac{1}{2})(\rho_{2}+\frac{1}{2})-(\rho_{1}+\rho_{2}+1)r_{1}}{2r_{1}-2\rho_{1}+N-2}N \\
\hspace{-2.4mm}& &\hspace{-2.4mm}
+\frac{\rho_{2}(2\rho_{2}+1)(2\rho_{1}+1)-r_{1}(2\rho_{2}(2\rho_{2}+1)-(2\rho_{1}+1))}{2(2r_{1}-2\rho_{1}+N-2)}
\bigg], \nonumber \\
E_{2,1}(x) \hspace{-2.4mm}&=&\hspace{-2.4mm} 
(x+r_1+\frac{1}{2})(x+\rho_2+\frac{N+1}{2})
\bigg[x^2+\frac{\rho_{1}\rho_{2}-(r_{1}+\frac{1}{2})(\rho_{1}+\rho_{2})}
{2r_1-2\rho_1+N+2}N \\
\hspace{-2.4mm}& &\hspace{-2.4mm}
+\frac{\rho_{1}(2\rho_{2}(2\rho_{2}+1)-(2r_{1}+1))-\rho_{2}(2\rho_{2}+1)(2r_{1}+1)}
{2(2r_{1}-2\rho_{1}+N+2)}, \nonumber \\
E_{3,1}(x) \hspace{-2.4mm}&=&\hspace{-2.4mm}
\bigg[x^2+\frac{\rho_{1}\rho_{2}+(r_{1}-\frac{1}{2})(\rho_{1}+\rho_{2})}
{2r_1+2\rho_1+4\rho_2+N-2}N \\
\hspace{-2.4mm}& &\hspace{-2.4mm}
+\frac{\rho_1((4\rho_2-1)(\rho_2+2r_1-1)-\rho_2)+\rho_{2}(2\rho_{2}-1)(2r_{1}-1)}
{2(2r_1+2\rho_1+4\rho_2+N-2)}\bigg] \nonumber \\
\hspace{-2.4mm}& &\hspace{-2.4mm}
\cdot
\bigg[(x+\frac{1}{2})^2+\frac{(\rho_{1}-\frac{1}{2})(\rho_{2}-\frac{1}{2})+r_1(\rho_{1}+\rho_{2}-1)}
{2r_1+2\rho_1+4\rho_2+N-2}N \nonumber \\
\hspace{-2.4mm}& &\hspace{-2.4mm}
+\frac{r_1((4\rho_2-1)(\rho_2+2\rho_1-1)-\rho_2)+\rho_{2}(2\rho_{2}-1)(2\rho_{1}-1)}
{2(2r_1+2\rho_1+4\rho_2+N-2)}\bigg], \nonumber 
\end{eqnarray}
\vspace{-1.2em}
\begin{eqnarray}
E_{4,1}(x) \hspace{-2.4mm}&=&\hspace{-2.4mm} 
(x-\rho_1)(x-\rho_2)(x+r_1+\frac{1}{2})(x+\rho_2+\frac{N+1}{2}), \\
E_{5,1}(x) \hspace{-2.4mm}&=&\hspace{-2.4mm}
\frac{(N-1)(2\rho_1+2\rho_2+N-1)^2(2r_1-2\rho_1+1)}{4(2\rho_1-2r_1+N-2)^2}
(x-\rho_2)(x+r_1+\frac{1}{2}), \\
E_{6,1}(x) \hspace{-2.4mm}&=&\hspace{-2.4mm} 
\frac{(N+1)(2r_1+2\rho_2-1)^2(2r_1-2\rho_1-1)}{4(2\rho_1-2r_1+N+2)^2}
(x-\rho_1)(x+\rho_2+\frac{N+1}{2}), \\
E_{7,1}(x) \hspace{-2.4mm}&=&\hspace{-2.4mm} 
\frac{(4\rho_2+N-1)^2(2r_1-2\rho_2+1)(2\rho_2-2\rho_1+N-1)}{4(4\rho_2-2\rho_1-2r_1+N-2)^2}
(x-\rho_1)(x+r_1+\frac{1}{2}), 
\end{eqnarray}
\begin{equation}
E_{8,1}(x) = 
\frac{(2r_1+2\rho_1-1)^2(2r_1-2\rho_2-1)(2\rho_2-2\rho_1+N+1)}{4(4\rho_2-2\rho_1-2r_1+N+2)^2}
(x-\rho_2)(x+\rho_2+\frac{N+1}{2}).
\end{equation}
It is not difficult to derive some sufficient conditions for $E_{d,1}(x)> 0$, 
where $x\in\chi^{(1)}_{d,1}$. 
Let us take the case $d=3$ for an example, from (4.26) we know that $E_{3,1}(x)> 0$ if 
$$
\frac{\rho_{1}\rho_{2}+(r_{1}-\frac{1}{2})(\rho_{1}+\rho_{2})}
{2r_1+2\rho_1+4\rho_2+N-2}N
+\frac{\rho_1((4\rho_2-1)(\rho_2+2r_1-1)-\rho_2)+\rho_{2}(2\rho_{2}-1)(2r_{1}-1)}
{2(2r_1+2\rho_1+4\rho_2+N-2)}
$$
and 
$$
\frac{(\rho_{1}-\frac{1}{2})(\rho_{2}-\frac{1}{2})+r_1(\rho_{1}+\rho_{2}-1)}
{2r_1+2\rho_1+4\rho_2+N-2}N
+\frac{r_1((4\rho_2-1)(\rho_2+2\rho_1-1)-\rho_2)+\rho_{2}(2\rho_{2}-1)(2\rho_{1}-1)}
{2(2r_1+2\rho_1+4\rho_2+N-2)}
$$
are both positive. These conditions are immediately satisfied if we assume that 
\begin{equation}
r_1, \rho_1, \rho_2>\frac{1}{2}.
\end{equation}

\par
For the other cases, we list some sufficient conditions below.
These conditions are obtained similarly by observing the right-hand sides of (4.24), (4.25), and (4.27)-(4.31).
\begin{eqnarray}
d=1: \hspace{-2.4mm}& &\hspace{-2.4mm} 0<r_1<\frac{(\rho_1+\frac{1}{2})(\rho_2+\frac{1}{2})}{\rho_1+\rho_2+1}, 
 -\frac{1}{2}<\rho_1<\rho_2+1 \text{ and } \rho_2>-\frac{1}{2}; 
\nonumber \\
d=2: \hspace{-2.4mm}& &\hspace{-2.4mm} \frac{N-3}{2}<r_1<\frac{N-2}{2}, \rho_1<0 \text{ and } 
-\frac{1}{2}<\rho_2<r_1-\frac{N-2}{2};  \nonumber \\
d=4: \hspace{-2.4mm}& &\hspace{-2.4mm} r_1>\frac{N-3}{2}, \rho_1<\frac{1}{2} \text{ and } \rho_2>-\frac{1}{2}; 
\nonumber \\
d=5: \hspace{-2.4mm}& &\hspace{-2.4mm} -1<r_1<0, \rho_1<r_1+\frac{1}{2} \text{ and } \rho_2>-\frac{1}{2};   \nonumber \\
d=6: \hspace{-2.4mm}& &\hspace{-2.4mm} r_1>-\frac{N-3}{4}, \rho_1<\rho_2 \text{ and } 
-\frac{N+1}{4}<\rho_2<-\frac{N-1}{4}; \nonumber \\
d=7: \hspace{-2.4mm}& &\hspace{-2.4mm} r_1>0, -\frac{1}{2}<\rho_1<\rho_2 \text{ and } 
r_1-\frac{1}{2}<\rho_2<r_1+\frac{1}{2}; \nonumber \\
d=8: \hspace{-2.4mm}& &\hspace{-2.4mm} r_1>\rho_2+\frac{1}{2}, \rho_1<\rho_2+\frac{N+1}{2} \text{ and } 
-\frac{N+3}{4}<\rho_2<-\frac{N-1}{4}. \nonumber
\end{eqnarray}
\end{rmk}

\begin{example} 
In this example we demonstrate the orthogonality of the 
exceptional Bannai-Ito polynomials explicitly. Consider the case ${\bf d}=(3,1)$, 
where the seed solution is given by $\phi_{(3,1)}=\xi_3(x)p^{(3)}_1(x)=\xi_3(x)B_1(x;-\rho_1,-\rho_2,r_1,r_2)$, 
and its eigenvalue is $\mu_{(3,1)}=r_1+r_2-1$.
According to Proposition 4.5, we can give the related exceptional Bannai-Ito polynomials:
$$
B^{(1)}_{(3,1),n}(x)=
(\lambda_n-\beta)\frac{\hat{\phi}_{(3,1)}^{(0)}(x)-\hat{\phi}_{(3,1)}^{(0)}(-x)}{x}B_n(x)
-(\mu_{(3,1)}-\beta)\frac{\hat{\phi}_{(3,1)}^{(0)}(x)}{x}(B_n(x)-B_n(-x)),
$$
where 
$\hat{\phi}_{(3,1)}^{(0)}(x)=\eta_3(x)p^{(3,1)}_{1}(x)=
(x-\rho_1)(x-\rho_2)B_1(x;-\rho_1,-\rho_2,r_1,r_2).$
Assume that $N$ is odd and $r_2=\rho_2+\frac{N}{2}$, 
then 
$\chi^{(1)}_{(3,1)}=\{\rho_2, -\rho_2-1,  \rho_2+1, \cdots, -\rho_2-\frac{N-1}{2}, \rho_2+\frac{N-1}{2}\}$ 
is the corresponding exceptional Bannai-Ito grid.
The discrete orthogonality (\ref{eq:orthogonality of XBI}) holds for the weight function 
\vspace{-2mm}$$
\hat{\omega}^{(1)}(x)=\frac{\sin(2\pi x)\Gamma(x-r_1+\frac{1}{2})\Gamma(-x-r_1+\frac{1}{2})\Gamma(x+\rho_1+1)\Gamma(-x+\rho_1)}
{\Gamma(x+r_2+\frac{1}{2})\Gamma(-x+r_2+\frac{1}{2})\Gamma(x-\rho_2+1)\Gamma(-x-\rho_2)y_1(x)y_2(x)},
$$
where $y_1(x)$, $y_2(x)$ are the polynomials
$$
y_1(x)=(r_{1}+r_{2}+\rho_{1}+\rho_{2}-1)x^{2}+
r_{1}\rho_{1}\rho_{2}+r_{2}\rho_{1}\rho_{2}+r_{1}r_{2}\rho_{1}+r_{1}r_{2}\rho_{2}-\rho_{1}\rho_{2}
$$
$$\hspace{-28mm}
-\frac{r_{1}\rho_{1}+r_{1}\rho_{2}+r_{2}\rho_{1}+r_{2}\rho_{2}}{2}+\frac{\rho_{1}+\rho_{2}}{4},
$$
$$
y_2(x)=(r_{1}+r_{2}+\rho_{1}+\rho_{2}-1)x^{2}+(r_{1}+r_{2}+\rho_{1}+\rho_{2}-1)x+
r_{1}\rho_{1}\rho_{2}+r_{2}\rho_{1}\rho_{2}+r_{1}r_{2}\rho_{1}+r_{1}r_{2}\rho_{2}
$$
$$\hspace{-26mm}
-r_{1}r_{2}-\frac{r_{1}\rho_{1}+r_{1}\rho_{2}+r_{2}\rho_{1}+r_{2}\rho_{2}-r_1-r_2}{2}+\frac{\rho_{1}+\rho_{2}-1}{4}.
$$
The orthogonality constant in the right-hand side of (\ref{eq:orthogonality of XBI}) is given by 
$h^{(1)}_{(3,1),n}=h^{(1)}_{(3,1),0}u^{(1)}_{1}\cdots u^{(1)}_{n}$, where
$h^{(1)}_{(3,1),0}=\sum_{x\in\chi^{(1)}_{(3,1)}}\hat{\omega}^{(1)}(x)B^{(1)}_{(3,1),0}(x)^2$, and
\begin{empheq}[left={u^{(1)}_{n}=\empheqlbrace}]{align}
-\frac{n(2r_1+2\rho_2+N-n)^2(2r_1-2\rho_1+N-n)(2r_1+2\rho_2+N-n-2)}
{8(2r_1-2\rho_1+N-2n)^2(\rho_1+\rho_2+\frac{n}{2}-1)}, \hspace{12mm} & \text{if n is even,} \nonumber \\
\scalebox{0.90}{$
\begin{aligned}
\frac{2(N-n)(2r_1-2\rho_1-n)(2r_1-2\rho_2-n)(\rho_1+\rho_2+\frac{n-1}{2})(\rho_1+\rho_2+\frac{n+1}{2})
(\rho_1-\rho_2-\frac{N-n}{2})}
{(2r_1-2\rho_1+N-2n)^2(2r_1+2\rho_2+N-n+1)(2r_1+2\rho_2+N-n-1)},  
\end{aligned}$}
& \text{ if n is odd.} \nonumber 
\end{empheq}
\end{example}
\vspace{0.1mm}

\section{Concluding remarks}
\par
This paper starts from the original idea that exceptional Dunkl shift operators can be obtained from 
the intertwining relations which always appear in the Darboux transformations. 
We call this method a generalized Darboux transformation (on first-order difference operators).
After the 1-step and the multiple-step exceptional Dunkl shift operators being successfully obtained through this method, 
we are able to give the exceptional Bannai-Ito operators with the restriction on certain coefficients. 
Especially, in this paper we mainly focus on the 1-step exceptional Bannai-Ito polynomials, 
which form the eigenpolynomials of the normalized 1-step exceptional Bannai-Ito operator. 

\par
In this generalized Darboux transformation the crucial role is played by 
the operator $\mathcal{F}_{\phi}$ which annihilates the seed solution $\phi(x)$. 
In fact, we should realize that 
the choice of the operator $\mathcal{F}_{\phi}$ in Section 2 is not unique. 
The only restriction we have used is that $\mathcal{F}_{\phi}$ is also a Dunkl shift operator. 
Without loss of generality, we may define $\mathcal{F}$ as
\begin{equation}
\label{eq:general opF}
\mathcal{F}=-\frac{R-I+f_1(x)}{\phi(-x)-\phi(x)+f_1(x)\phi(x)}+\frac{TR-I+f_2(x)}{\phi(-x-1)-\phi(x)+f_2(x)\phi(x)}
\end{equation}
such that $\mathcal{F}[\phi(x)]=0$, and correspondingly, 
$$
\chi(x)=\phi(x)-\phi(-x)-f_1(x)\phi(x), \hspace{2mm}
\tilde{\chi}(x)=\phi(-x-1)-\phi(x)+f_2(x)\phi(x).
$$
In the case regarding $\mathcal{F}_{\phi}$, we have let 
\begin{equation}
f_1(x)=0, \hspace{2mm}
f_2(x)=(\beta(-x-1)+\beta(x))/\beta(x).
\end{equation}
Recall that the 1-step exceptional Dunkl shift operator has the form  
$$
H^{(1)}=\alpha^{(1)}(x)(R-I)+\beta^{(1)}(x)(TR-I)+\gamma^{(1)}(x),
$$
then from the intertwining relation 
$
\mathcal{F}\circ H=H^{(1)}\circ\mathcal{F}
$
one can obtain the coefficients of $H^{(1)}$ by comparing the coefficients of the operators
$RTR$, $T$, $TR$, $R$, $I$ appeared in each side. 
Specifically,  from the coefficients of $RTR$ and $T$ we have
\begin{eqnarray}
\alpha^{(1)}(x) \hspace{-2.4mm}&=&\hspace{-2.4mm}
\frac{\beta(-x)[\phi(x-1)-\phi(-x)+f_2(-x)\phi(-x)]}{\phi(x)-\phi(-x)-f_1(x)\phi(x)}, \\
\beta^{(1)}(x) \hspace{-2.4mm}&=&\hspace{-2.4mm}
\frac{\alpha(-x-1)[\phi(-x-1)-\phi(x+1)-f_1(-x-1)\phi(-x-1)]}{\phi(-x-1)-\phi(x)+f_2(x)\phi(x)}.
\end{eqnarray}
After substituting (5.3), (5.4) into the remaining three equations with respect to the coefficients 
of $TR$, $R$, $I$ we can derive three different expressions of $\gamma^{(1)}(x)$ 
(which are omitted in view of their length). 
Of course, these three expressions of $\gamma^{(1)}(x)$ must be equal to each other, 
hence one can derive the restrictions regarding $f_1(x)$ and $f_2(x)$ from this fact. 
However, this is not an easy task as it seems to be, since the equations with respect to $\gamma^{(1)}(x)$ are 
complicated and difficult to solve. 
Three other feasible cases we have found are:
\vspace{-2mm}
\begin{eqnarray}
f_1(x)=0, \hspace{-2mm}& &\hspace{-2mm} f_2(x)=0; \\
f_1(x)=(\alpha(-x)+\alpha(x))/\alpha(x), \hspace{-2mm}& &\hspace{-2mm} f_2(x)=0;  \\
f_1(x)=(\alpha(-x)+\alpha(x))/\alpha(x), \hspace{-2mm}& &\hspace{-2mm} f_2(x)=(\beta(-x-1)+\beta(x))/\beta(x).
\end{eqnarray}

\par
As for the general cases of $f_1(x)$ and $f_2(x)$, 
we shall leave it as an open problem for the readers of interest.

\par
If one look at the 1-step exceptional eigenfunction $\mathcal{F}[B_n(x)]$, 
the following two expressions can be obtained similarly as we did in the proof of Lemma 4.1:
\begin{eqnarray*}
\beta(x)\chi(x)\tilde{\chi}(x)\mathcal{F}[B_n(x)] 
\hspace{-2.4mm}&=&\hspace{-2.4mm} 
\left(\lambda_n+\alpha(x)f_1(x)+\beta(x)f_2(x)\right)B_n(x)(\phi(-x)-\phi(x)) \\
\hspace{-2.4mm}& &\hspace{-2.4mm}
-\left(\mu+\alpha(x)f_1(x)+\beta(x)f_2(x)\right)\phi(x)(B_n(-x)-B_n(x)) \\
\hspace{-2.4mm}& &\hspace{-2.4mm}
+(\lambda_n-\mu)f_1(x)B_n(x)\phi(x), \\
\alpha(x)\chi(x)\tilde{\chi}(x)\mathcal{F}[B_n(x)] \hspace{-2.4mm}&=&\hspace{-2.4mm} 
\left(\mu+\alpha(x)f_1(x)+\beta(x)f_2(x)\right)\phi(x)(B_n(-x-1)-B_n(x)) \\
\hspace{-2.4mm}& &\hspace{-2.4mm} 
-\left(\lambda_n+\alpha(x)f_1(x)+\beta(x)f_2(x)\right)B_n(x)(\phi(-x-1)-\phi(x)) \\
\hspace{-2.4mm}& &\hspace{-2.4mm} 
+(\mu-\lambda_n)f_2(x)B_n(x)\phi(x).
\end{eqnarray*}
Recall that for the coefficients $\alpha(x)$, $\beta(x)$ of the Bannai-Ito operator, 
$\alpha(-x)+\alpha(x)$ and $\beta(-x-1)+\beta(x)$ are both constants (see (2.23)), 
hence $\alpha(x)f_1(x)+\beta(x)f_2(x)$ is always constant in the four cases with respect to 
(5.2), (5.5), (5.6), (5.7). 
Thus, one can derive 1-step exceptional Bannai-Ito polynomials from the normalizations of these expressions. 
It turns out (5.2), (5.5), (5.6), (5.7) actually lead to different exceptional polynomials. 
Discussions on the orthogonality with respect to (5.5), (5.6), (5.7) can be made in the same manner as that of (5.2). 
In view of this fact, there are more than one type of exceptional Bannai-Ito polynomials. 
This is the nontrivial aspect of our ``generalized" Darboux transformation, due to the non-uniqueness of $\mathcal{F}$.
\vspace{-1em}

\section*{Acknowledgements}
\par
The authors thank the referees for their thorough reading and helpful suggestions.

\vspace{-1em}
\appendix
\section{}
\par
For readers' convenience, we list some data for the coefficients 
$\mu_{d,m}-\beta$, $\lambda_n-\mu_{d,m}$ and $C_{d,m,n}$, 
which have been used in the proof of Proposition 4.3.

\par
For $m$ is even, we have 
\vspace{-2mm}\begin{eqnarray}
\mu_{1,m}-\beta \hspace{-2.4mm}&=&\hspace{-2.4mm} 
\frac{m}{2}-r_1-r_2, \\
\mu_{2,m}-\beta \hspace{-2.4mm}&=&\hspace{-2.4mm} 
\frac{m}{2}-\rho_1-\rho_2, \\
\mu_{3,m}-\beta \hspace{-2.4mm}&=&\hspace{-2.4mm} 
\frac{m}{2}-r_1-r_2-\rho_1-\rho_2, \\
\mu_{4,m}-\beta \hspace{-2.4mm}&=&\hspace{-2.4mm} 
\frac{m}{2},
\end{eqnarray}
for $m$ is odd, we have 
\vspace{-2mm}\begin{eqnarray}
\mu_{5,m}-\beta \hspace{-2.4mm}&=&\hspace{-2.4mm} 
\frac{m}{2}-r_2-\rho_1, \\
\mu_{6,m}-\beta \hspace{-2.4mm}&=&\hspace{-2.4mm} 
\frac{m}{2}-r_1-\rho_2, \\
\mu_{7,m}-\beta \hspace{-2.4mm}&=&\hspace{-2.4mm} 
\frac{m}{2}-r_2-\rho_2, \\
\mu_{8,m}-\beta \hspace{-2.4mm}&=&\hspace{-2.4mm} 
\frac{m}{2}-r_1-\rho_1.
\end{eqnarray}

\par
For $m$ is odd and $n$ is odd, we have 
\begin{eqnarray}
\lambda_n-\mu_{1,m} \hspace{-2.4mm}&=&\hspace{-2.4mm} 
\frac{m-n}{2}, \\
\lambda_n-\mu_{2,m} \hspace{-2.4mm}&=&\hspace{-2.4mm} 
r_1+r_2-\rho_1-\rho_2+\frac{m-n}{2}, \\
\lambda_n-\mu_{3,m} \hspace{-2.4mm}&=&\hspace{-2.4mm} 
-\rho_1-\rho_2+\frac{m-n}{2}, \\
\lambda_n-\mu_{4,m} \hspace{-2.4mm}&=&\hspace{-2.4mm} 
r_1+r_2+\frac{m-n}{2}, 
\end{eqnarray}
for $m$ is even and $n$ is odd, we have 
\begin{eqnarray}
\lambda_n-\mu_{5,m} \hspace{-2.4mm}&=&\hspace{-2.4mm} 
r_1-\rho_1+\frac{m-n}{2}, \\
\lambda_n-\mu_{6,m} \hspace{-2.4mm}&=&\hspace{-2.4mm} 
r_2-\rho_2+\frac{m-n}{2}, \\
\lambda_n-\mu_{7,m} \hspace{-2.4mm}&=&\hspace{-2.4mm} 
r_1-\rho_2+\frac{m-n}{2}, \\
\lambda_n-\mu_{8,m} \hspace{-2.4mm}&=&\hspace{-2.4mm} 
r_2-\rho_1+\frac{m-n}{2}.
\end{eqnarray}

\par
For $m$ is even and $n$ is even, we have 
\begin{eqnarray}
C_{1,m,n} \hspace{-2.4mm}&=&\hspace{-2.4mm} 
-\frac{(r_1+r_2-\frac{m}{2})(r_1+r_2-\frac{n}{2})(\frac{m-n}{2})(r_1+r_2-\rho_1-\rho_2)}
{(r_1+r_2-\rho_1-\rho_2-m)(r_1+r_2-\rho_1-\rho_2-n)}, \\
C_{2,m,n} \hspace{-2.4mm}&=&\hspace{-2.4mm} 
\frac{(\rho_1+\rho_2-\frac{m}{2})(r_1+r_2-\frac{n}{2})(r_1+r_2-\rho_1-\rho_2+\frac{m-n}{2})(r_1+r_2-\rho_1-\rho_2)}
{(r_1+r_2-\rho_1-\rho_2+m)(r_1+r_2-\rho_1-\rho_2-n)}, \\
C_{3,m,n} \hspace{-2.4mm}&=&\hspace{-2.4mm} 
\frac{(r_1+r_2+\rho_1+\rho_2-\frac{m}{2})(r_1+r_2-\frac{n}{2})(\rho_1+\rho_2-\frac{m-n}{2})(r_1+r_2-\rho_1-\rho_2)}
{(r_1+r_2+\rho_1+\rho_2-m)(r_1+r_2-\rho_1-\rho_2-n)}, \\
C_{4,m,n} \hspace{-2.4mm}&=&\hspace{-2.4mm} 
\frac{(-\frac{m}{2})(r_1+r_2-\frac{n}{2})(r_1+r_2+\frac{m-n}{2})(r_1+r_2-\rho_1-\rho_2)}
{(r_1+r_2+\rho_1+\rho_2+m)(r_1+r_2-\rho_1-\rho_2-n)},
\end{eqnarray}
for $m$ is odd and $n$ is even, we have 
\begin{eqnarray}
C_{5,m,n} \hspace{-2.4mm}&=&\hspace{-2.4mm} 
\frac{(r_2+\rho_1-\frac{m}{2})(r_1+r_2-\frac{n}{2})(r_1-\rho_1+\frac{m-n}{2})(r_1+r_2-\rho_1-\rho_2)}
{(r_1-r_2-\rho_1+\rho_2+m)(r_1+r_2-\rho_1-\rho_2-n)}, \\
C_{6,m,n} \hspace{-2.4mm}&=&\hspace{-2.4mm} 
-\frac{(r_1+\rho_2-\frac{m}{2})(r_1+r_2-\frac{n}{2})(r_2-\rho_2+\frac{m-n}{2})(r_1+r_2-\rho_1-\rho_2)}
{(r_1-r_2-\rho_1+\rho_2-m)(r_1+r_2-\rho_1-\rho_2-n)}, \\
C_{7,m,n} \hspace{-2.4mm}&=&\hspace{-2.4mm} 
\frac{(r_2+\rho_2-\frac{m}{2})(r_1+r_2-\frac{n}{2})(r_1-\rho_2+\frac{m-n}{2})(r_1+r_2-\rho_1-\rho_2)}
{(r_1-r_2+\rho_1-\rho_2+m)(r_1+r_2-\rho_1-\rho_2-n)}, \\
C_{8,m,n} \hspace{-2.4mm}&=&\hspace{-2.4mm} 
-\frac{(r_1+\rho_1-\frac{m}{2})(r_1+r_2-\frac{n}{2})(r_2-\rho_1+\frac{m-n}{2})(r_1+r_2-\rho_1-\rho_2)}
{(r_1-r_2+\rho_1-\rho_2-m)(r_1+r_2-\rho_1-\rho_2-n)}.
\end{eqnarray}


\end{document}